\documentclass{article}
\usepackage[utf8]{inputenc}
\usepackage[legalpaper, margin=1.3in]{geometry}

\usepackage{times}
\usepackage{epsfig}
\usepackage{graphicx}
\usepackage{amsmath}
\usepackage{amssymb}
\usepackage{tabularx} 
\usepackage{mathtools} 
\usepackage{amsfonts}
\usepackage{amsthm}
\usepackage{booktabs}
\usepackage{authblk}
\usepackage{tikz-cd}
\usepackage[ruled,vlined,linesnumbered]{algorithm2e}
\usepackage[sort]{cite}
\usepackage{microtype}
\usepackage{appendix}
\usepackage{enumerate}
\usepackage{babel}
\addto{\captionsenglish}{}

\addto\captionsspanish{

}

\newcommand{\PP}{{\mathbb{P}}}
\newcommand{\CC}{{\mathbb{C}}}
\newcommand{\RR}{{\mathbb{R}}}
\newcommand{\ZZ}{{\mathbb{Z}}}
\newcommand{\Gr}{{\rm Gr}}

\usepackage{chngcntr}
\usepackage{apptools}
\AtAppendix{\counterwithin{theorem}{section}}

\usepackage[pagebackref=true,breaklinks=true,colorlinks,bookmarks=false]{hyperref}
\usepackage{cleveref}

\newtheorem*{theorem*}{Theorem}
\newtheorem{theorem}{Theorem}[section]
\newtheorem{definition}[theorem]{Definition}
\newtheorem{proposition}[theorem]{Proposition}

\newtheorem{lemma}[theorem]{Lemma}

\newtheorem{innercustomgeneric}{\customgenericname}
\providecommand{\customgenericname}{}
\newcommand{\newcustomtheorem}[2]{
  \newenvironment{#1}[1]
  {
   \renewcommand\customgenericname{#2}
   \renewcommand\theinnercustomgeneric{##1}
   \innercustomgeneric
  }
  {\endinnercustomgeneric}
}

\newcustomtheorem{customtheorem}{Theorem}
\newcustomtheorem{customproposition}{Proposition}
\newcustomtheorem{customlemma}{Lemma}
\newcustomtheorem{customcorollary}{Corollary}

\begin{document}
\title{Theoretical and Numerical Analysis of 3D Reconstruction Using Point and Line Incidences}
\author[1]{Felix Rydell}
\author[2]{Elima Shehu}
\author[3]{Angélica Torres}
\affil[1]{{\small KTH Royal Institute of Technology, Stockholm, Sweden}}
\affil[2]{{\small University of Osnabr\"uck, Osnabr\"uck, Germany}}
\affil[2]{{\small Max Planck Institute for Mathematics in the Sciences, Leipzig, Germany}}
\affil[3]{{\small Centre de Recerca Matemàtica, Barcelona, Spain}}

\maketitle

\begin{abstract}

  We study the joint image of lines incident to points, meaning the set of image tuples obtained from fixed cameras observing a varying 3D point-line incidence. We prove a formula for the number of complex critical points of the triangulation problem that aims to compute a 3D point-line incidence from noisy images. Our formula works for an arbitrary number of images and measures the intrinsic difficulty of this triangulation. Additionally, we conduct numerical experiments using homotopy continuation methods, comparing different approaches of triangulation of such incidences. In our setup, exploiting the incidence relations gives a notably faster point reconstruction with comparable accuracy.

\end{abstract}

\section*{Introduction}

The Structure-from-Motion pipeline aims to estimate the camera positions and the relative position of the objects present in a given set of images. The process starts by identifying point and line features in one image that are recognizable as the same points or lines in another. These matched features, that we call \textit{correspondences}, are used to estimate the camera position and orientation, which then allow for the triangulation step where the position of the points and lines is estimated.  

In this work, we study the triangulation of points and lines satisfying a certain incidence relation. Specifically, we focus on the triangulation of points contained in a single line; see \Cref{fig:L1_and_P1}. We note that our methods can also be used when multiple lines going through a fixed point. We develop the theory for the triangulation of these problems for an arbitrary number of pinhole cameras assuming complete visibility. 

The triangulation of points is well understood and, in practice, it is efficiently implemented. However, to the best of our knowledge, incidence relations are not considered in these implementations. The inclusion of line features and incidence relations in the triangulation process could give a more accurate estimation of scenes as lines are more robust to noise than points, and incidence relations appear frequently in interior scenes and man-made scenarios; see \cite{Schindler2006LineBasedSF}, \cite{Huang2018LearningTP}, \cite{Liu_2023_CVPR}. Additionally, in some real data sets standard feature detection algorithms fail due to a lack of point correspondences but succeed when line correspondences are taken into account \cite{fabbri2020trplp}.

\begin{figure}
    \centering
    \includegraphics[width = 0.50\textwidth]{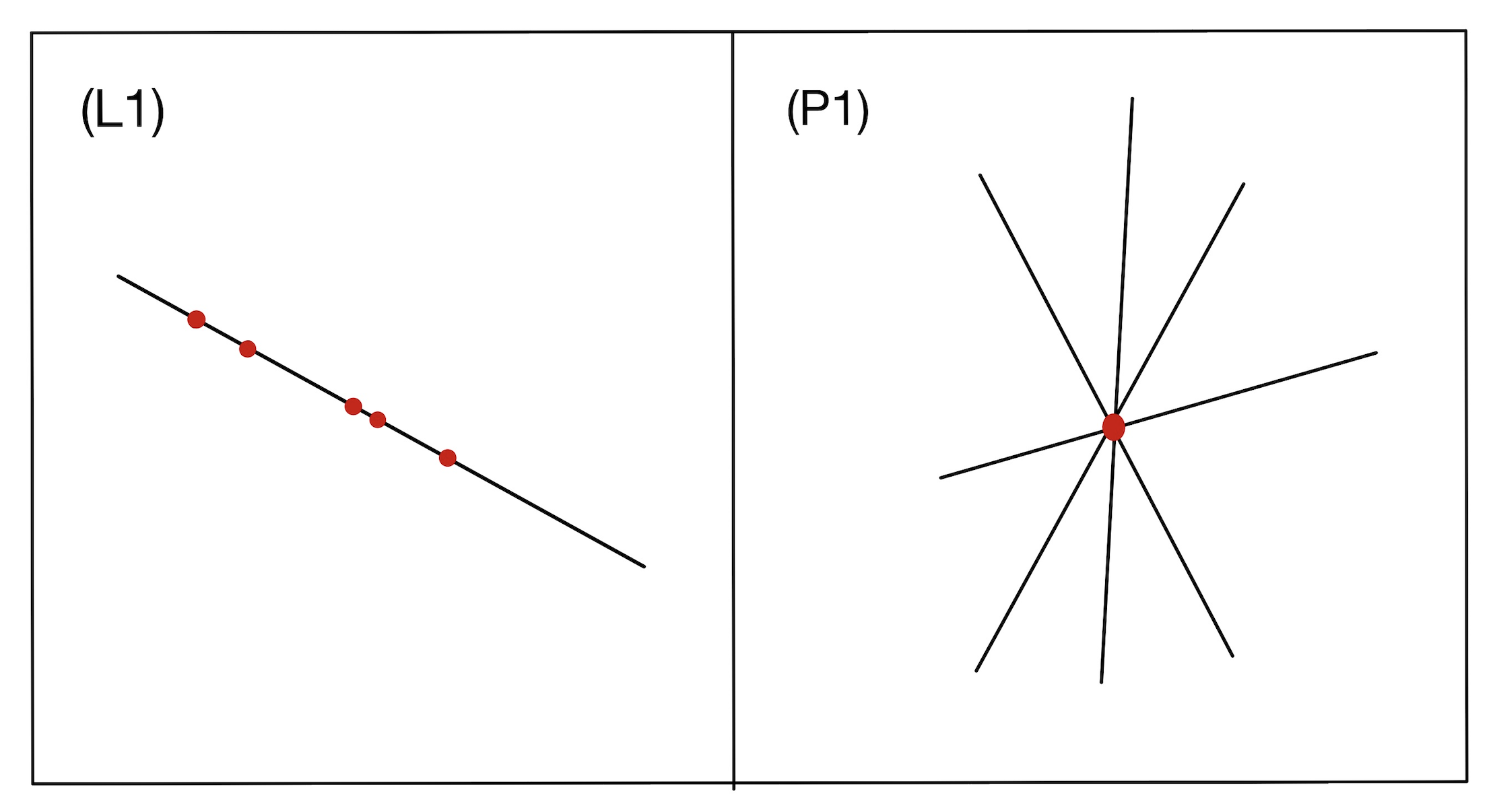}
    \caption{The illustration of two different types of incidence relations: (L1)~represents the scenario where multiple points are incident to a line, while (P1) represents the scenario where multiple lines are incident to a point.}
    \label{fig:L1_and_P1}
\end{figure}

The main tools in our work are \textit{(algebraic) varieties}, that is, the vanishing sets of systems of polynomial equations. Algebraic varieties have been used extensively to study triangulation of point correspondences \cite{agarwal2019ideals, EDDegree_point, faugeras1995geometry}, line correspondences \cite{faugeras1993three,breiding2022line, Kileel_Thesis} and minimal problems \cite{duff2019plmp,Duff20_partialvis} to name a few. In these works, algebraic varieties arise naturally due to the algebraic nature of pinhole cameras.

A \textit{pinhole camera}, is modeled by a (complex) projective linear map $C: \mathbb{P}^3 \dashrightarrow \mathbb{P}^2$ defined by a $3\times 4$ matrix~$C$ of full rank that takes a point $X\in\PP^3$ and sends it to $CX\in\PP^2.$ 
A camera arrangement with $m\geq 2$ cameras is denoted by $\mathcal{C}=(C_1,\ldots ,C_m)$, and the \textit{joint camera map}
\begin{align}
\begin{split}
\Phi_\mathcal{C}:\mathbb{P}^3&\dashrightarrow (\mathbb{P}^2)^m,\\
    X&\mapsto (C_1X,\ldots,C_mX),
    \end{split}
\end{align}
models the process of taking the images of a point $X$ in homogeneous coordinates with $m$ cameras. For fixed cameras, the \textit{(point) multiview variety} $\mathcal{M}_\mathcal{C}$ is the smallest variety that contains all point correspondences. An extensive account of the pinhole cameras is given by \cite{Hartley2004}, and a survey of the multiview variety is found in \cite{trager2015joint}. The joint camera map can be extended from the space of points $\PP^3$ to the space of lines, denoted by $\Gr(1,\PP^3)$, where each line is parametrized as the span of two points. The map
\begin{align}
  \begin{split}
\Upsilon_{\mathcal{C}}:\mathrm{Gr}(1,\PP^3)&\dashrightarrow (\PP^2)^m,\\
  \quad L&\mapsto (C_1\cdot L,\ldots, C_m\cdot L).  
  \end{split}
\end{align}
models the image of a line~$L$ taken by the cameras of~$\mathcal C$. To be precise, if $u,v$ span the line $L$ in $\PP^3$, then $C\cdot L$ is defined as $\ell=Cu\times Cv\in \PP^2$, where $\times$ is the cross product. Observe that $\{x\in \PP^2:\ell^Tx=0\}$ equals $\mathrm{span}\{Cu,Cv\}$. Recently, in \cite{breiding2022line}, the authors study the \textit{line multiview variety} referred to as $\mathcal{L}_\mathcal{C}$, which is the smallest variety containing all line correspondences.

Our main contribution is the definition and study of the \textit{anchored} point and line multiview varieties. Given a fixed line $L$ in $\mathbb P^3$, the anchored point multiview variety, $\mathcal{M}_\mathcal{C}^L$, is defined as the smallest variety containing all point correspondences coming from points in $L$; and the anchored line multiview variety, $\mathcal{L}_\mathcal{C}^X$, is the smallest variety containing the line correspondences coming from lines passing through a given point $X\in\PP^3$. 

For these new varieties we prove formulas for their \textit{Euclidean distance degree} (EDD), which are a measurement of the complexity for error correction using exact algebraic methods \cite{draisma2016euclidean}. Specifically, we prove the following theorem:
\begin{theorem*} Let $\mathcal{C}$ be a generic arrangement of $m$ cameras.  
\begin{enumerate}
    \item $\mathrm{EDD}( \mathcal{M}_\mathcal{C}^L)=3m-2 $.
    \item If $m\ge 3$, then $\mathrm{EDD}( \mathcal{L}_\mathcal{C}^X)=\frac{9}{2}m^2-\frac{19}{2}m+3.$
\end{enumerate}
\end{theorem*}
We provide a precise definition of this degree in \Cref{ss: EDD}. Previous work on EDD for multiview varieties includes \cite{EDDegree_point,harris2018chern} and~\cite[Section 5]{breiding2022line}. The EDD of the point multiview variety is  $\mathrm{EDD}(\mathcal{M}_\mathcal{C})=\frac{9}{2}m^3-\frac{21}{2}m^2+8m-4,$ according to \cite{EDDegree_point}. The fact that the EDDs of the anchored multiview varieties are polynomials of smaller degrees suggests that they are less complex for the purpose of data correction. This conclusion is backed by the results of our numerical experiments using \texttt{HomotopyContinuation.jl}  \cite{HC.jl}, presented in \Cref{s: NUM}. 

Another contribution of our work is the numerical simulations for the triangulation of points contained in a line using the anchored multiview varieties. We present different approaches to triangulating data of type (L1) that differ from the traditional method of fitting point correspondences to the multiview variety $\mathcal{M}_\mathcal{C}$. For~$m=2$ views, our implementation is notably faster than the traditional triangulation of points described above, while the accuracy is comparable. For $m=3$ views we get both higher accuracy and faster speed by using $\mathcal{L}_\mathcal{C}$, compared to usual point triangulation. All of our proposed methods outperform the traditional approaches in terms of run-time and give a comparably accurate result. In practice, special software and hardware are used for triangulation, and based on our experiments, we believe that these approaches could be implemented efficiently with good results.

\subsection*{Related work} 

\paragraph{Algebra.} Algebraic geometry, whose connection to computer vision is well established, is our main tool for the theoretical study of the triangulation of problem (L1). Fundamental theory regarding the point multiview variety is found in \cite{faugeras1995geometry, Hartley1997, Hartley2004}, in particular for two and three views. The paper \cite{trager2015joint} by Trager et. al. serves as a survey for the algebraic properties of this multiview variety. For applications, finding a good set of polynomial constraints satisfied by the point correspondences is important, this is precisely the work of \cite{agarwal2019ideals} by Agarwal et al.

Regarding the algebra of lines in a computer vision setting, Kileel presented in \cite[Definition 3.9, Theorem 3.10]{Kileel_Thesis} several types of multiview varieties with respect to a point-line incidence relation in $\PP^3$ in 3 views, and their basic properties. One of those types is called the LLL-multiview variety and constitutes a special case of the more recent work by Breiding et. al. \cite{breiding2022line}, where they study algebraic properties of the line multiview variety. Furthermore, a recent manuscript \cite{2303.02066} studies the polynomial constraints satisfied by line correspondences.

\paragraph{Projective Reconstruction.} Different approaches to the triangulation of points have been considered in the literature, in particular for two views \cite{hartley1997triangulation,kanazawa1995reliability,kanatani2005statistical,kanatani2008triangulation,beardsley1994navigation,beardsley1997sequential}. Hartley and Sturm compare many different approaches in \cite{hartley1997triangulation}, including the ``midpoint'' method. The midpoint method inputs a point correspondence in two views and outputs the midpoint of the line segment determined by the points on the two back-projected lines that are closest to each other. In other words, it finds the midpoint of the common perpendicular to the two back-projected lines. This is a natural approach, but it has several downsides, also pointed out in \cite{beardsley1994navigation,beardsley1997sequential}. Importantly, it is not \textit{projectively invariant}. A projectively invariant reconstruction has the property that acting on the cameras by a global projective transformation induces the same action on the triangulated point or line. Instead, Hartley and Sturm propose minimizing the reprojection error as the optimal triangulation method, which is accepted as a standard formulation \cite{stewenius2005hard} and it is projectively invariant. In \cite{stewenius2005hard}, the focus is on triangulation via minimizing reprojections errors in three views and the authors point out that three views often lead to greater stability and stronger disambiguation compared to two views.

\paragraph{Implementation of Line Reconstruction.} From the algorithmic point of view, the simultaneous re\-cons\-truction of point and line features have been studied specially with the goal of reconstructing line segments. For example, \cite{BARTOLI2005416} provides a thorough overview of the structure-from-motion pipeline using lines, going through different methods for line parameterization and error correction depending on such parameterizations. In \cite{Quan}, Quan and Takeo study the algebraic structure of line correspondences with uncalibrated affine cameras. They reduce the problem of reconstructing affine lines to the reconstruction of projective points in a lower-dimensional projective camera. In~\cite{Hartley1997}, Hartley and coauthors provide an algorithm for the reconstruction of point and line features where 3D lines are parameterized by their projections in 2 views (as the intersection of two back-projected planes).
Micusik and Wildenauer reconstruct lines to estimate line segments, but only use incident points for error correction~\cite{Micusik_Wildenauer}. Furthermore in \cite{L_2_ref}, the authors propose a technique for 3D reconstruction incorporating line segments.  
They assert that this method surpasses traditional approaches, especially in efficiency while getting accurate
results. Finally, in \cite{L_3_ref}, the authors present a framework for the computation of the relative motion between two images using a triplet of lines. 

\bigskip

This paper is structured as follows. In the first part of \Cref{s: AMV} we formally introduce anchored multiview varieties and study some properties. In the second part of \Cref{s: AMV}, we study their smoothness and find their Euclidean Distance Degree. Finally, in \Cref{s: NUM}, we provide a numerical analysis of different approaches to reconstructing point correspondences incident to a line correspondence. The proofs of all our results are included in the Supplementary Material together with a small background on the concepts of Algebraic Geometry used for these results and the code for the numerical experiments. The code used for the numerical experiments can be found in the GitHub repository \url{github.com/amtorresbu/Anchored_LMV.git}.

\paragraph{Acknoledgements.} The authors thank Kathlén Kohn for helpful discussions, Lukas Gustafsson for the initial insight of \Cref{thm: anchoredmulti-isomorphism}, and Viktor Larsson for pointing out relevant references at the start of this project. 
 Felix Rydell and Angélica Torres were supported by the Knut and Alice Wallenberg Foundation within their WASP (Wallenberg AI, Autonomous
Systems and Software Program) AI/Math initiative. Elima Shehu is funded by the Deutsche Forschungsgemeinschaft (DFG, German Research Foundation), Projektnummer 445466444. Angélica Torres is currently supported by the Spanish State Research Agency, through the Severo Ochoa and María de Maeztu Program for Centers and Units of Excellence in R\&D.

\section{Anchored Multiview Varieties} \label{s: AMV}

A \textit{camera} refers to a full-rank $3\times 4$ matrix. A \textit{camera arrangement} is a collection $\mathcal{C}=(C_1,\ldots ,C_m)$ of $m\ge 2$ cameras whose \textit{centers}, i.e. kernels, are all distinct.

A \textit{variety} is the solution set to a system of polynomial equations. The \textit{Zariski closure} of a set $U$ is the smallest variety containing $U$. We work in complex projective space, denoted $\PP^n$. This is defined as $(\CC^n\setminus\{0\})/\sim$, where $x\sim y$ if $x$ and $y$ differ by a non-zero constant. The set of lines in $\PP^n$ is denoted as $\Gr(1,\PP^n)$. In the language of algebraic geometry, it is called the \textit{Grassmanian of lines} in $\PP^n$. In general we denote with upper case letters world objects and with lower case letters image objects. For a rational map $\varphi:X\dashrightarrow Y$, we denote $\varphi|_A$ the restriction of $\varphi$ to $A\subseteq X$.

\begin{definition}\label{def:AMV}
       Let $X\in\PP^3$ be a point distinct from all camera centers, and $L$ a line in $\PP^3$ 
       containing none of the camera centers.
       \begin{enumerate}
           \item The \textnormal{anchored point multiview variety}, denoted $\mathcal{M}_\mathcal{C}^L$, is defined as the Zariski closure of the image of the map
           \begin{equation}
               \begin{split}
                   \Phi_\mathcal{C}|_L: L & \dashrightarrow (\PP^2)^m, \\
                   X & \longmapsto (C_1X,\ldots ,C_mX).
               \end{split}
           \end{equation}
           \item The \textnormal{anchored line multiview variety}, denoted $\mathcal{L}_\mathcal{C}^X$, is defined as the Zariski closure of the image of the map
           \begin{equation}
               \begin{split}
                   \Upsilon_\mathcal{C}|_{\Lambda(X)}: \Lambda(X) & \dashrightarrow (\PP^2)^m, \\
                   L & \longmapsto (C_1\cdot L,\ldots ,C_m\cdot L),
               \end{split}
           \end{equation}
           where $\Lambda(X)$ denotes the set of lines in $\PP^3$ 
           that contain $X$ and $\ell_i=C_i\cdot L$ is defined as in the introduction. 
       \end{enumerate}
\end{definition}
The anchored point multiview variety~$\mathcal{M}_\mathcal{C}^L$ is the smallest variety that contains all point correspondences $\Phi_{\mathcal C}(X)$ for $X\in L$. Similarly, the anchored line multiview variety $\mathcal{L}_\mathcal{C}^X$ is the smallest variety that contains all line correspondences $\Upsilon_{\mathcal C}(L)$ for $L$ meeting $X$ and no center. We highlight that our definition of anchored multiview varieties is different from the anchored features in\cite{Sola2012} for monocular EKF-SLAM.

\begin{proposition}\label{prop: first-eq}
    Consider an arrangement of $m$ cameras $\mathcal{C}=(C_1,\ldots,C_m)$, a point $X\in\PP^3$ and a line $L$ in $\mathbb P^3$ satisfying the conditions of~\Cref{def:AMV}. 
    \begin{enumerate}
        \item If there are two different camera centers $c_i$ and $c_j$ such that the span of $\{c_i,c_j,L\}$ is $\PP^3$, then 
        \begin{equation}
            \mathcal{M}_\mathcal{C}^L=\{(x_1,\ldots,x_m)\in \mathcal{M}_{\mathcal{C}}: x_i\in C_i\cdot L\}.
        \end{equation}

        \item If for each camera center $c_i$, the line spanned by $c_i$ and $X$ does not contain any other camera center, then 
        \begin{equation}
            \mathcal{L}_\mathcal{C}^X=\{(\ell_1,\ldots,\ell_m)\in \mathcal{L}_{\mathcal{C}}: C_iX\in \ell_i\}.
        \end{equation}
    \end{enumerate}
\end{proposition}
The proofs for this and all our results are found in the Supplementary Material.


\subsection{Properties of Anchored Multiview Varieties} 

A fundamental property of the anchored multiview varieties is that they are linearly isomorphic to multiview varieties arising from projections $\PP^1\dashrightarrow\PP^1$ and~$\PP^2\dashrightarrow\PP^1,$ respectively. A linear isomorphism is a linear map (given by a matrix) with an inverse that is also linear. 

Consider an arrangement $\widetilde{\mathcal{C}}$ of $m$ full-rank $2\times 2$ matrices, and an arrangement $\widehat{\mathcal{C}}$ of $m$ full-rank $2\times 3$ matrices. Also here we assume $m\ge 2$. We define $\mathcal{M}_\mathcal{\widetilde{C}}^{1,1}$, and $\mathcal{M}_\mathcal{\widehat{C}}^{2,1}$, 
respectively as the Zariski closure of the image of the 
maps
    \begin{equation}
        \begin{split}
            \Phi_{\widetilde{\mathcal C}}: \PP^1&\longrightarrow (\PP^1)^m, \\
            X & \longmapsto (\widetilde C_1X,\ldots,\widetilde C_mX)
        \end{split}
    \end{equation}
and 
    \begin{equation}
        \begin{split}
            \Phi_{\widehat{\mathcal C}}:\PP^2&\dashrightarrow (\PP^1)^m,\\
            X& \longmapsto (\widehat{C}_1X,\ldots \widehat{C}_mX).
        \end{split} 
    \end{equation}

\begin{theorem}\label{thm: anchoredmulti-isomorphism}  $ $ 

\begin{enumerate}
     \item Let $\phi_L:L \to \PP^1$ and $\psi_{\mathcal C,i}:C_i\cdot L\to \PP^1$ be any choices of linear isomorphisms. Let $\widetilde{\mathcal{C}}$ denote the arrangement of matrices $\widetilde{C}_i:=\psi_{\mathcal{C},i} \circ C_i\circ \phi_L^{-1} $. Then 
    \begin{align}
        \psi_{\mathcal{C},L}:=(\psi_{\mathcal C,1},\ldots,\psi_{\mathcal C,m}):\mathcal{M}_\mathcal{C}^L\to  \mathcal{M}_{\widetilde{\mathcal{C}}}^{1,1}
    \end{align}
    is a linear isomorphism.
    \item Let $\phi_X: \Lambda(X)\to \PP^2$ and $\psi_{\mathcal C,i}: \Lambda(C_i X)\to \PP^1$ be any choices of linear isomorphisms. Let $\widehat{\mathcal{C}}$ denote the arrangement of matrices $\widehat{C}_i:=\psi_{\mathcal C,i} \circ C_i\circ \phi_X^{-1} $. Then 
    \begin{align}
        \psi_{\mathcal{C},X}:=(\psi_{\mathcal C,1},\ldots,\psi_{\mathcal C,m}): \mathcal{L}_\mathcal{C}^X\to \mathcal{M}_{\widehat{\mathcal{C}}}^{2,1}
    \end{align}
    is a linear isomorphism.
\end{enumerate}
\end{theorem} 
In \Cref{sss: Red}, we exploit these linear isomorphisms to improve the speed of triangulation for our setting. A consequence of this theorem is also that many results for the anchored multiview varieties translate into results about $\mathcal{M}_{\widetilde{\mathcal{C}}}^{1,1}$ and $\mathcal{M}_{\widehat{\mathcal{C}}}^{2,1}$. 

\begin{proposition}\label{prop: dim} $\mathcal{M}_\mathcal{C}^L$ and $\mathcal{L}_\mathcal{C}^X$ are irreducible. Further, 
\begin{enumerate}
    \item $\mathcal{M}_\mathcal{C}^L$ is isomorphic to $\PP^1$. In particular, $\dim\mathcal{M}_\mathcal{C}^L=1$.
    \item If the span of the centers $c_i$ and the point $X$ are not collinear, then $\dim \mathcal{L}_\mathcal{C}^X=2$. 
\end{enumerate}
\end{proposition}

Using the fundamental matrices of $C_i$ and $C_j$, denoted as $F^{ij}$, we introduce a set of polynomial constraints that must be fulfilled by the correspondences of points or lines in the anchored multiview varieties. These constraints imply that some of the equations obtained from~\Cref{prop: first-eq} are redundant in the generic case. 

\begin{proposition}\label{prop: eq} For a point $X\in \PP^3$ and line $L$ in $\PP^3$, let $\mathcal{C}$ be a generic (random) camera arrangement of $m$ cameras. 
\begin{enumerate}
    \item $x\in \mathcal{M}_\mathcal{C}^L$ if and only if $x_1^TF^{1j}x_j=0$ for every $j=2,\ldots,m$ and $x_i^TC_i\cdot L=0$ for every $i=1,\ldots,m$.
    \item $\ell\in \mathcal{L}_\mathcal{C}^X$ if and only if
    \begin{align}
    \begin{split}
        &\det \begin{bmatrix}
    C_1^T \ell_1 & C_2^T \ell_2 & C_i^T\ell_i 
    \end{bmatrix}=0,\\  
    &\det \begin{bmatrix}
    C_1^T \ell_1 & C_3^T \ell_3 & C_i^T\ell_i 
    \end{bmatrix}=0
    \end{split}
    \end{align}
    for $i=3,\ldots,m$ and $\ell_i^T C_iX=0$ for every $i=1,\ldots,m$.
\end{enumerate}
\end{proposition}

 The determinantal constraints described in the second item correspond to the constraints that are satisfied by elements of the line multiview variety, presented in~\cite{breiding2022line}.

In Supplementary Material Section B we define the \textit{multidegree} of a variety in a product of projective spaces, determine it for the anchored multiview varieties and explain its relevance for computer vision.


\subsection{The Euclidean Distance problem} \label{ss: EDD} 

For a variety $\mathcal{X}\subseteq \RR^n$ and a point $u\in\RR^n$ outside the variety, a natural problem is to find the closest point on $\mathcal X$ to $u$, which corresponds to the optimization problem
\begin{equation}\label{eq :ED_problem}
    \mathrm{minimize}\quad \sum_{i=1}^n(u_i-x_i)^2\quad \textnormal{subject to} \quad x\in \mathcal{X}.
\end{equation}
\Cref{eq :ED_problem} is called the \textit{Euclidean distance problem} and models the process of error correction and fitting noisy data to a mathematical model $\mathcal{X}$. In the case of a smooth variety, defined below, the \textit{Euclidean distance degree} (EDD) is the number of complex solutions to the critical equations of \eqref{eq :ED_problem} \cite{draisma2016euclidean}. The EDD is an estimate of how difficult it is to solve this problem by exact algebraic methods. For a variety $\mathcal X$ in a product of projective spaces $\PP^{n_1}\times \cdots \times \PP^{n_m}$, the EDD is the EDD of $\mathcal X\cap U_1\times \cdots \times U_m\subseteq \RR^{n_1+\cdots +n_m}$, where $U_i$ is a generic affine patch of $\PP^{h_i}$ for each $i$. An \textit{affine patch} $U$ of $\PP^n$ is a subset defined by an affine equation with non-zero constant part, for instance $x_0=1$. We have $U\cong \RR^n$ over the real numbers.

A variety $\mathcal{X}$ is \textit{smooth} at a point $x$ if $\mathcal{X}$ locally around~$x$ looks like Euclidean space. $\mathcal X$ is \textit{smooth} if all its points are smooth. The \textit{singular locus} of a variety is the set of non-smooth points. See \cite{Gathmann} for more details. The singular locus of the point multiview variety is well-understood, see for instance \cite{trager2015joint}, and it is mostly understood for the line multiview variety, see \cite[Section 3]{breiding2022line}. The proposition below guarantees that the anchored multiview varieties are smooth for generic camera arrangements, which helps us to compute its EDD.  

\indent
\begin{proposition}\label{prop: smooth} $ $

\begin{enumerate}
    \item $\mathcal{M}_\mathcal{C}^L$ smooth.
    \item If there are exactly two cameras, or the centers together with the point $X$ span $\PP^3$, then $\mathcal{L}_\mathcal{C}^X$ is smooth. 
\end{enumerate}
\end{proposition}

For the EDD to be relevant in a particular setting, the Euclidean distance has to be a good measurement of distance. This is naturally the case for points in $\RR^2$, i.e. an affine patch of $\PP^2$. This extends to an affine patch of $\mathcal M_{\mathcal C}$. It is perhaps less clear how to best measure distances between lines. Recall that we regard $\mathcal L_{\mathcal C}$ as a subset of $(\PP^2)^m$. Indeed, we identify image lines with the point that defines its normal vector. We choose the Euclidean distance between these normal vectors (in affine patches of $\PP^2$) as our distance between lines.

The theorem below present our computations of the EDDs of the anchored multiview varieties. \\Additionally, our numerical computations using the \texttt{HomotopyContinuation.jl}\cite{breiding2018homotopycontinuation} package in the \texttt{julia}\cite{bezanson2012julia} programming language have confirmed the accuracy of these formulas for $m\le 10$.

\begin{theorem}\label{thm: EDDthm} Let $\mathcal{C}$ be a generic arrangement of $m$ cameras.  
\begin{enumerate}
    \item $\mathrm{EDD}( \mathcal{M}_\mathcal{C}^L)=3m-2 $.
    \item If $m\ge 3$, then $\mathrm{EDD}( \mathcal{L}_\mathcal{C}^X)=\frac{9}{2}m^2-\frac{19}{2}m+3.$
\end{enumerate}
\end{theorem}

We end with an implication of \Cref{thm: anchoredmulti-isomorphism}: There is a direct correspondence between the EDDs of the anchored multiview varieties and $\mathcal{M}_{\widetilde{\mathcal{C}}}^{1, 1},\mathcal{M}_{\widehat{\mathcal{C}}}^{2, 1}$.

\begin{theorem}\label{cor: EDD-p} Let $\widetilde{\mathcal{C}}$ and $\widehat{\mathcal{C}}$ be generic arrangements of cardinality $m$.
\begin{enumerate}
    \item $\mathrm{EDD}( \mathcal{M}_{\widetilde{\mathcal{C}}}^{1, 1})=3m-2 $.
    \item If $m\ge 3$, then $\mathrm{EDD}( \mathcal{M}_{\widehat{\mathcal{C}}}^{2,1})=\frac{9}{2}m^2-\frac{19}{2}m+3.$
\end{enumerate}
\end{theorem}

\section{Numerical Experiments} \label{s: NUM}

In this section we conduct numerical experiments of the triangulation process of incident point and lines. All the code can be found in the Supplementary Material.

By (L1) we refer to point-line arrangements in $\PP^3$ of one line incident to $p$ points, as depicted in~\Cref{fig:L1_and_P1}. Our goal is to reconstruct such arrangements, given point correspondences incident to a single line correspondence across $m$ views. The following is a list of natural approaches for such triangulation given a camera arrangement $\mathcal{C}$. It is not necessarily a complete list. 

\begin{enumerate}[leftmargin =0.5em]
    \item[(L1).0] Triangulate each point correspondence by fitting it to the point multiview variety $\mathcal{M}_{\mathcal{C}}$, i.e. find the closest point correspondence in $\mathcal{M}_{\mathcal{C}}$;
    \item[(L1).1] Reconstruct the 3D line $L$ by back-projecting the image lines from two views. Triangulate the point correspondences by fitting them to the anchored multiview variety $\mathcal{M}_{\mathcal{C}}^L$;
    \item[(L1).2] Triangulate two point correspondences by fitting them to $\mathcal{M}_{\mathcal{C}}$ to get 3D points $X$ and $Y$. Let $L$ be the line they span in $\RR^3$. Triangulate the remaining point correspondences by fitting them to $\mathcal{M}_{\mathcal{C}}^L$;  
    \item[(L1).3] Triangulate one point correspondence by fitting it to $\mathcal{M}_{\mathcal{C}}$ to get a 3D point $X$. Reconstruct the 3D line $L$ by fitting the line correspondence to the anchored variety $\mathcal{L}_\mathcal{C}^X$. Triangulate the remaining point correspondences by fitting them to $\mathcal{M}_{\mathcal{C}}^L$;  
    \item[(L1).4] Reconstruct the 3D line $L$ by fitting the line correspondence to the line multiview variety $\mathcal{L}_\mathcal{C}$. Triangulate the point correspondences by fitting them to the anchored multiview variety $\mathcal{M}_{\mathcal{C}}^L$.
\end{enumerate}  

Approaches (L1).1-4 take the incidence relations into account in the triangulation, whereas (L1).0 does not. Therefore, in contrast to (L1).1-4, the resulting triangulation of (L1).0 does not preserve the point-line incidences; see \Cref{fig:my_label1}. 
\begin{figure}
   \centering
   \includegraphics[width = 0.65\textwidth]{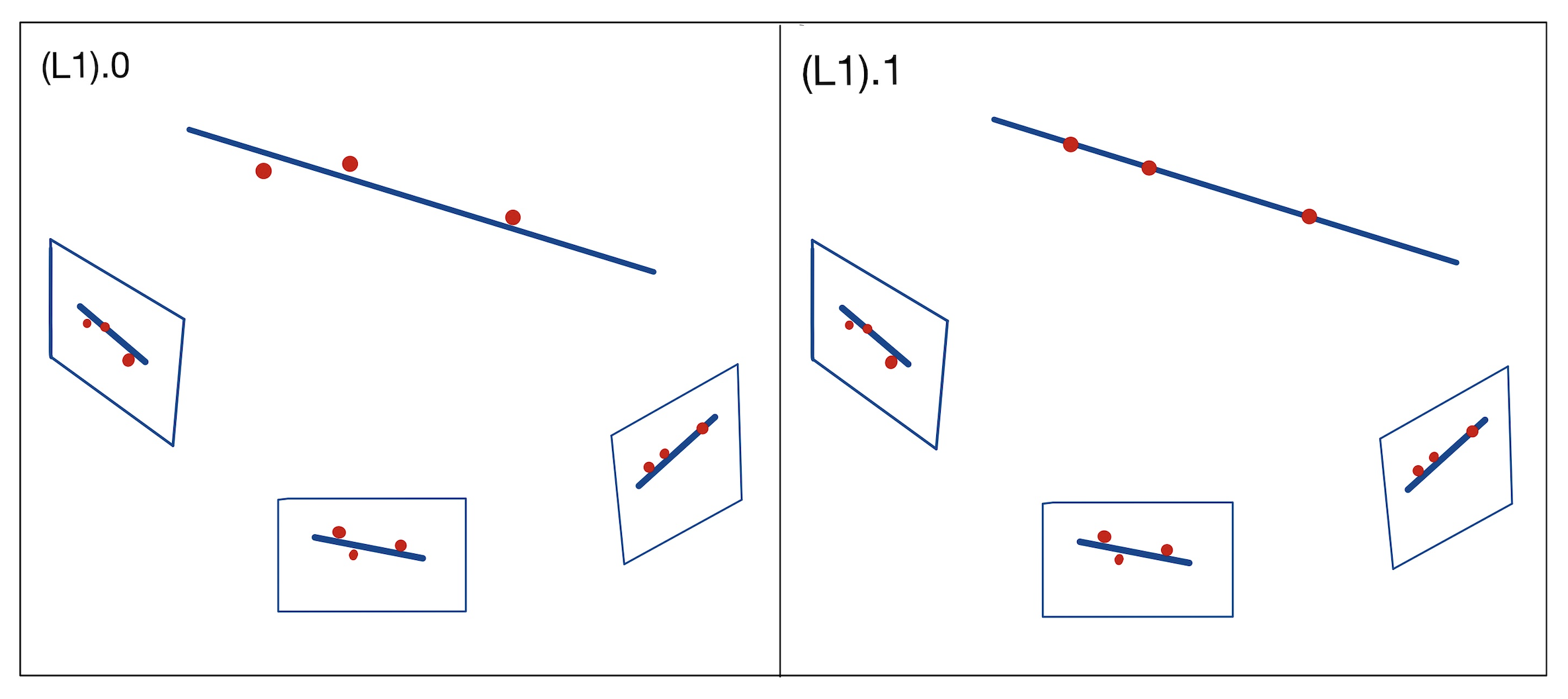}
   \caption{Comparison of (L1).0 and (L1).1 triangulation approaches. (L1).0 expects non-collinear points in the reconstruction, while (L1).1 produces collinear points.}
   \label{fig:my_label1}
\end{figure}

\subsection{Implementation}

Our experiments are implemented in \texttt{HomotopyContinuation.jl} \cite{breiding2018homotopycontinuation} in \texttt{Julia} \cite{bezanson2012julia}. Here we explain them in some detail. The full explanation can be found in Supplementary Material. We ran the code on a Intel(R) Core(TM) i5-8300H CPU running at 2.30GHz.

\subsubsection{Reducing the number of parameters} \label{sss: Red} We reduce the number of parameters involved in the problem of fitting data to $\mathcal{M}_\mathcal{C}^L$, respectively $\mathcal{L}_\mathcal{C}^X$, by translating it to and solving it for $\mathcal{M}_{\widetilde{\mathcal{C}}}^{1,1}$, respectively $\mathcal{M}_{\widehat{\mathcal{C}}}^{2,1}$. Details are found in Supplementary Material Section B. This translation corresponds to a reduction of parameters because an affine patch of $\mathcal{M}_{\widetilde{\mathcal{C}}}^{1,1}$ lives in $(\RR^1)^m$, while an affine patch of $\mathcal{M}_{\mathcal{C}}^{L}$ lives in $(\RR^2)^m$. The analogous is true for $\mathcal{M}_{\widehat{\mathcal{C}}}^{2,1}$ and~$\mathcal{L}_{\mathcal{C}}^X$. 

We call the methods that do \textit{not} use this translation \textit{standard} and for instance, write (L1).1 std. to denote it. By just writing (L1).1, we denote the implementation of this translation.

\subsubsection{Evaluating error} 

In our experiments, each iteration starts by randomly generating a 3D line $L$ and $p$ points $X_i$ on this line. This line and these points are projected by randomly generated cameras to line and point correspondences in $(\RR^2)^m$. For fixed~$\epsilon =10^{-12}$, randomly generated noise vectors $\sigma(\epsilon)$ of length $\epsilon$ are added in each factor. Our approaches then find points $Y_i\in \RR^3$ that match the noisy correspondences. 
We measure the accuracy by taking the logarithm after averaging the relative error of reconstructed points:
\begin{equation}
e= \log_{10}\left(\frac{\sum \|Y_i-X_i\|}{p\epsilon}\right),
\end{equation}
where $\|Y_i-X_i\|$ denotes the Euclidean distance. The interpretation of this number is that the error in the data gets amplified by $10^e$ during the triangulation. 

\subsubsection{Algorithms} 

The pseudocodes for all approaches are included in the Supplementary Material. The pseudocode for approach (L1).1 is presented in~\Cref{alg:L1.1}. We use the notation that for a column vector $X\in \RR^n$, $[X; \; 1]\in \RR^{n+1}$ is the vector we get by adding a $1$ as the last coordinate. Let $\iota$ be the function that scales a vector such that its last coordinate is 1, and then removes that coordinate. In lines 1 and 2 of the algorithm, we generate noisy data by introducing randomly generated noise $\sigma(\epsilon)$ as explained above, where $\epsilon=10^{-12}$ is a fixed parameter for our experiments. Lines 3 and 4 correspond to the triangulation of the line and the point correspondences using the anchored point multiview variety at the line $L_0$. Note that we solve the closest point problem in line 4 by computing the zeros of a system of polynomial (critical) equations using the \textit{solve} function in \texttt{HomotopyContinuation.jl} \cite{HC.jl}. Finally, line 5 compares the points obtained in the previous step with the original starting points, measuring the logarithmic average relative error of the $p$ points. 

Most other pseudocodes are similar to~\Cref{alg:L1.1}. For instance, approach (L1).0 does not take into account the line, so it corresponds deleting lines 1 and 3 from~\Cref{alg:L1.1} and modifying
the minimization domain in line 4, replacing $L_0$ by $\mathcal{M}_{\mathcal{C}}$. In approach (L1).2, line 3 is modified so that $L_0$ is obtained by the span of two triangulated points instead of by intersecting two back-projected planes.

 \begin{algorithm}
\SetKwInOut{Input}{Input}\SetKwInOut{Output}{Output}
\SetKwInOut{Return}{Return}
\caption{One iteration of the (L1).1 std. method given a randomly generated camera arrangement $\mathcal{C}$ of $3\times 4$ matrices, a projective line $L$ spanned by two vectors of $\RR^4$, and $p$ points $X_i\in \RR^3$ such that $[X_i;\; 1]$ lie on $L$.}
\label{alg:L1.1}
\Input{$\mathcal{C}=(C_1,\ldots,C_m)$, $L$, $X_1,\ldots,X_p$}
\Output{The log of the average relative error} 
  \For{$j$ \textnormal{from $1$ to $m$}}{
    \For{$i$ \textnormal{from $1$ to $p$}}{
        $q_{i,j} \gets \iota(C_j[X_i;\;1]) +\sigma(\epsilon$)\;}
        $u_j \gets \iota(C_j \cdot L) +\sigma(\epsilon)$\;}
    $L_0\gets \mathrm{nullspace} \begin{bmatrix}
    C^T_1[u_1;\;1]&C^T_2[u_2;\;1]
    \end{bmatrix}^T$\;
    \For{$i$ from 1 to $p$}{$Y_i \gets \underset{{X\in \RR^3: [X;\;1]\in L_0}}{\mathrm{argmin}} \sum^{m}_{j=1} (q_{i,j}-\iota(C_j[X;\;1]))^2$\; }
    $e\gets \log_{10}\left(\frac{1}{p\epsilon}\sum_{i=1}^p \|Y_i-X_i\|\right)  $\;

\Return{$e$}

\end{algorithm}

\subsection{Results}

The main results of our numerical experiments are presented in \Cref{fig: m=2,fig: m=3,fig: m=4} and~\Cref{tab: m=2,tab: m=3,tab: m=4}. We compare the performance of the five different triangulation approaches for $p=5$ and different number of $m$ cameras. In the tables, we present the median, mean and standard deviation~$\sigma$ of the logarithmic average relative error and time. The results are displayed in histograms and tables created with \texttt{Plots.jl}~\cite{plots}. 

Note that for two cameras, (L1).0 and (L1).2 are the exact same, and (L1).1 and (L1).4 are essentially the same. In \Cref{fig: m=2} and \Cref{tab: m=2} we present results for $m=2$ and therefore only include (L1).0, (L1).1 and (L1).3. We also compare with the standard implementations of (L1).1 and (L1).3 that do not use the linear isomorphisms of \Cref{thm: anchoredmulti-isomorphism} to improve computation time as described in \Cref{sss: Red}. This simulation is iterated 1000 times. 

 \Cref{fig: m=3} and \Cref{tab: m=3} compare the different reconstruction approaches for $m=3$ cameras. The simulations are again iterated 1000 times.

\Cref{fig: m=4} and \Cref{tab: m=4} compare the different reconstruction approaches for $m=4$ cameras. The simulations are iterated only 100 times, due to the longer run-time of the simulations.

Our theoretical studies in \Cref{s: AMV}, specifically~\Cref{thm: EDDthm}, guarantee that methods such as (L1).1, (L1).2 and (L1).3 are less complex from the algebraic point of view, than (L1).0 for triangulation of points incident to a line. Additionally, using \Cref{sss: Red}, we were able to improve the computation speed while retaining comparable accuracy, which is exemplified by the two cases (L1).1 vs. (L1).1 std. and (L1).3 vs. (L1).3 std. in \Cref{tab: m=2}.

Finally, our numerical simulations show that in our \texttt{HomotopyContinuation.jl} implementation, the (L1).1 approach is the fastest across all number of views, but least accurate. In the case of $m=3$, (L1).4 is the most accurate and faster than the traditional (L1).0 method. Even for $m=4$, (L1).4 is the most accurate but is notably slower than (L1).0. Increasing the number of cameras $m$ improves accuracy and reduces speed, but the exact impact depends on the approach. Among (L1).1-4, (L1).1 is least affected in both regards, and (L1).4 is most affected in both regards. The speed can be thought of as an affine function in the number of point correspondences $p$. Reconstructing the line correspondence takes some fixed amount of time independent of $p$ and then all $p$ correspondences are independently reconstructed.

\begin{figure}
\begin{center}
    \includegraphics[width = 0.45\textwidth]{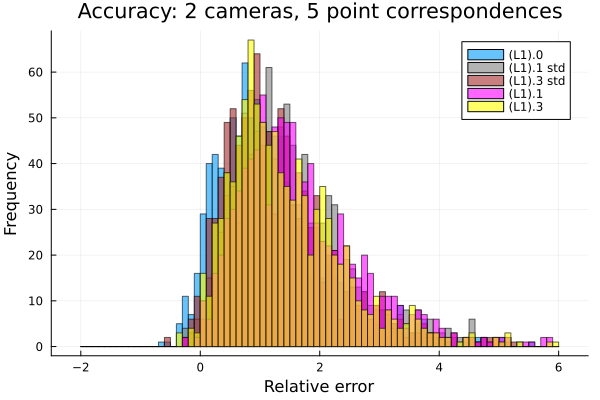}
    \includegraphics[width = 0.45\textwidth]{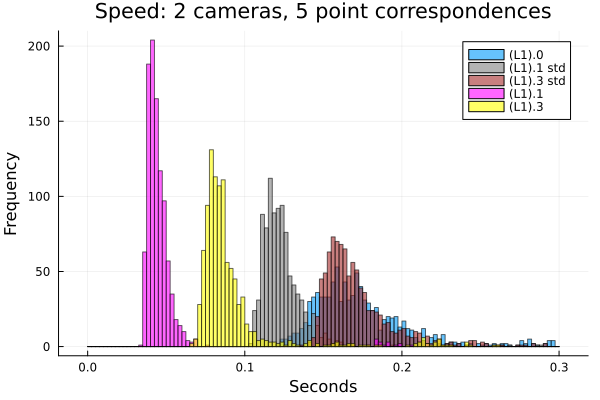}
    \caption{Triangulation of $p=5$ point correspondences incident to a line for $m=2$ views with complete visibility. 1000 iterations. The histograms show the frequency of the average relative error and the running time of the depicted triangulation methods. }
    \label{fig: m=2}
\end{center}
\end{figure}

\begin{table}
\begin{center}
\begin{tabular}{|l|c|c|c|l|} 
\hline
Accuracy & median & mean & $\sigma$\\
\hline\hline
(L1).0 & 1.031 & 1.243 & 1.213 \\
(L1).1 & 1.489 & 1.702 & 1.085 \\
(L1).1 std. & 1.374 & 1.548 & 0.990\\
(L1).3 & 1.250 & 1.449 & 0.964 \\
(L1).3 std. & 1.147 & 1.393 & 1.205 \\
\hline\hline
Speed & median & mean & $\sigma$\\
\hline\hline
(L1).0 & 0.170 & 0.201 & 0.220 \\
(L1).1 & 0.0430 & 0.0470 & 0.0207 \\
(L1).1 std. & 0.122 & 0.141 & 0.202\\
(L1).3 & 0.0838 & 0.0941 & 0.0464 \\
(L1).3 std. & 0.167 & 0.188 & 0.0764 \\
\hline
\end{tabular}
\end{center}
\caption{Triangulation of $p=5$ point correspondences incident to a line for $m=2$ views with complete visibility. 1000 iterations. The tables show the accuracy and speed of the depicted triangulation methods in terms of median, mean and standard deviation. }
\label{tab: m=2}
\end{table}

\begin{figure}
\begin{center}
    \includegraphics[width = 0.45\textwidth]{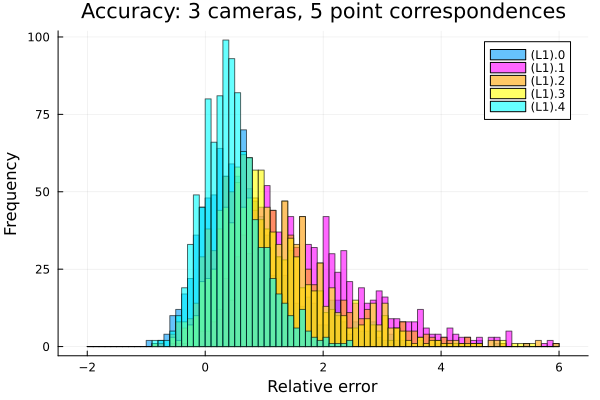}
    \includegraphics[width = 0.45\textwidth]{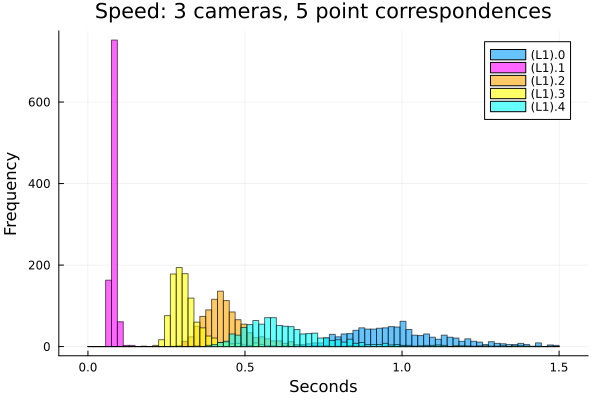}
    \caption{Triangulation of $p=5$ point correspondences incident to a line for $m=3$ views with complete visibility. 1000 iterations. The histograms show the frequency of the average relative error and the running time of the depicted triangulation methods. }
    \label{fig: m=3}
\end{center}
\end{figure}

\begin{table}
\begin{center}
\begin{tabular}{|l|c|c|c|l|} 
\hline
Accuracy & median & mean & $\sigma$\\
\hline\hline
(L1).0 & 0.620 & 0.817 & 1.023 \\
(L1).1 & 1.562 & 1.781 & 1.128 \\
(L1).2 & 0.878 & 1.125 & 0.998\\
(L1).3 & 1.011 & 1.243 & 1.078 \\
(L1).4 & 0.407 & 0.499 & 0.901 \\
\hline\hline
Speed & median & mean & $\sigma$\\
\hline\hline
(L1).0 & 0.984 & 1.060 & 0.514 \\
(L1).1 & 0.0806 & 0.0858 & 0.0318 \\
(L1).2 & 0.431 & 0.456 & 0.155\\
(L1).3 & 0.304 & 0.331 & 0.134 \\
(L1).4 & 0.616 & 0.763 & 0.500 \\
\hline
\end{tabular}
\end{center}
\caption{Triangulation of $p=5$ point correspondences incident to a line for $m=3$ views with complete visibility. 1000 iterations. The tables show the accuracy and speed of the depicted triangulation methods in terms of median, mean and standard deviation. }
\label{tab: m=3}
\end{table}

\begin{figure}
\begin{center}
    \includegraphics[width = 0.45\textwidth]{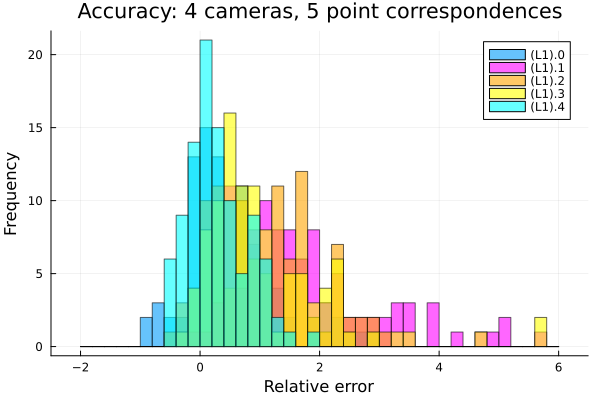}
    \includegraphics[width = 0.45\textwidth]{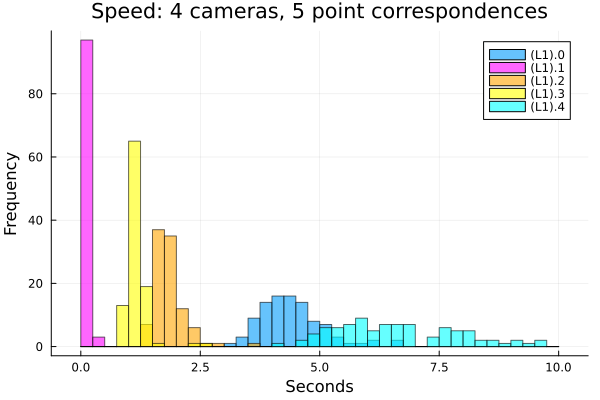}
    \caption{Triangulation of $p=5$ point correspondences incident to a line for $m=4$ views with complete visibility. 100 iterations. The histograms show the frequency of the average relative error and the running time of the depicted triangulation methods.}
    \label{fig: m=4}
\end{center}
\end{figure}

\begin{table}
\begin{center}
\begin{tabular}{|l|c|c|c|l|} 
\hline
Accuracy & median & mean & $\sigma$\\
\hline\hline
(L1).0 & 0.385 & 0.551 & 0.776 \\
(L1).1 & 1.437 & 1.762 & 1.249 \\
(L1).2 & 1.095 & 1.077 & 0.845\\
(L1).3 & 0.817 & 1.110 & 1.094 \\
(L1).4 & 0.203 & 0.421 & 1.324 \\
\hline\hline
Speed & median & mean & $\sigma$\\
\hline\hline
(L1).0 & 4.365 & 4.675 & 1.576 \\
(L1).1 & 0.141 & 0.149 & 0.0431 \\
(L1).2 & 1.809 & 1.863 & 0.323\\
(L1).3 & 1.107 & 1.157 & 0.234 \\
(L1).4 & 6.599 & 7.391 & 3.068 \\
\hline
\end{tabular}
\end{center}
\caption{Triangulation of $p=5$ point correspondences incident to a line for $m=4$ views with complete visibility. 100 iterations. The tables show the accuracy and speed of the depicted triangulation methods in terms of median, mean and standard deviation. }
\label{tab: m=4}
\end{table}

\section{Conclusion and Future Work}

This work studied the triangulation of incident points and lines. We introduced the anchored line multiview varieties and showed that they correspond to multiview varieties arising from projections $\PP^2\dashrightarrow\PP^1$ and $\PP^1\dashrightarrow\PP^1$. In addition, we proved that they are less complex for reconstruction via critical points. These theoretical results are also aligned with the numerical experiments we conducted. 
In particular, the proposed methods in~\Cref{s: NUM} compare different methods for triangulating a set of points incident to a line. We highlight that according to our experiments, the use of the theoretical results in~\Cref{s: AMV} allows for a notably faster triangulation while preserving the accuracy of the traditional method for point reconstruction. 

We hope this work motivates the use of new algebraic constraints in the triangulation of incident points and lines. Some questions that arose in the development of this work, and that we believe are worth studying, are included below.
\begin{itemize}
    \item A triangulation approach inspired by \cite{BARTOLI2005416} consists of reconstructing the 3D line $L$ by fitting it to $ \mathcal L_{\mathcal C}$ solving the following optimization problem 
    \begin{align}
        \mathrm{minimize}\quad \sum_{ij}d(x_{ij},\ell_i)^2 \textnormal{subject to} \quad \ell\in \mathcal{L}_\mathcal{C}.
    \end{align}
    Here, $x_{ij}$ denotes the $j=1,\ldots,p$ point correspondences across $i=1,\ldots,m$ cameras, and $d(x_{ij},\ell_i)$ is the affine distance from a line in $\RR^2$ to the point, meaning
\begin{align}
d(x_{ij},\ell_i)=\frac{|1+(\ell_{i})_1(x_{ij})_1+(\ell_{i})_2(x_{ij})_2|}{\sqrt{(\ell_{i})_1^2+(\ell_{i})_2^2}}.
\end{align} We believe that this method is more accurate than (L1).4, but with longer running time, due to the appearance of fractions.

\item Consider the point-line problem (P1) defined by one point incident to $l$ lines, illustrated in~\Cref{fig:L1_and_P1}. Can similar approaches to (L1).1-4 be formulated for the triangulation of the (P1) setting? Can they be used to improve accuracy and/or the speed of triangulation?

\item Our implementation for the numerical experiments does not directly correspond to the specialized software and hardware used for triangulation in practice. However, our results motivate that patterns displayed in our figures and tables could translate to such settings. This could improve the efficiency of current specialized solvers. 
\end{itemize}

{\small
\bibliographystyle{ieeetr}
\bibliography{VisionBib}
}
\newpage

\appendix

\section*{Appendix}  In this Supplementary Material, we prove all the mathematical results from the main body of the paper. For convenience of the
reader, we start in \Cref{s:AlgPre} by explaining the elementary notions of algebraic geometry, that are helpful to understand the rest of the Supplementary Material. Results that appear in the main body
of the paper are restated and given the same numbering. Additional results not stated in the main body are numbered independently.

\Cref{s: AMV App} deals with \Cref{s: AMV} apart from the Euclidean distance degree. In \Cref{s: EDDprelim} we provide helpful background for the EDD calculations that are carried out in \Cref{s: EDD}. In \Cref{s: Pseudo} we provide pseudocode for the different reconstruction approaches from \Cref{s: NUM}.


\section{Algebraic Geometry Preliminaries}\label{s:AlgPre} 
The \textit{complex projective space} of dimension $n$ is the set of one-dimensional linear subspaces of $\CC^{n+1}$, equivalently 
$\PP^n\coloneqq (\CC^{n+1}\setminus \{0\})/\sim $, where $\sim$ denotes the equivalence relation defined by
\begin{equation}
    x\sim y \Leftrightarrow x=\lambda y \quad\mbox{ for some }\quad 0\neq\lambda\in\CC.
\end{equation}

The ring of polynomials in $n+1$ variables is denoted by $R\coloneqq \CC\left[x_0,\ldots ,x_n\right]$. A subset $X\subseteq \PP^n$  is called a \textit{projective algebraic variety}, when there exists a collection $\{f_1,\ldots ,f_k\}$ of homogeneous polynomials such that $X=\{x\in\PP^n \mid f_1(x)=\cdots f_k(x)=0\}$. In other words, $X$ is the vanishing set of the polynomials $f_i$ for $i=1,\ldots ,k$ such  that each term of the polynomial has degree $d$ and $f_i(\lambda x_0,\ldots ,\lambda x_n)=\lambda^df_i(x_0,\ldots ,x_n)$. Similarly, a subset $X\subseteq \PP^{n_1}\times \cdots \times \PP^{n_m}$ is an algebraic variety, if $X$ is the vanishing set of multi-homogeneous polynomials $\{f_1,\ldots ,f_k\}$.  

The Zariski topology on $\PP^n$ (or $\PP^{n_1}\times \cdots \times \PP^{n_m}$) is the topology whose closed sets are algebraic varieties. Therefore, given a set $U$, its Zariski closure, denoted $\overline{U}$, is the smallest variety containing $U$.

Let $X\subseteq \PP^n$ and $Y\subseteq \PP^m$ be projective varieties. A map $\varphi:X\longrightarrow Y$ is regular if it can be written as 
    \begin{equation}
        \varphi(x)=\left[\varphi_0(x) : \cdots :\varphi_m(x) \right]
    \end{equation}
for some polynomials $\varphi_0, \ldots ,\varphi_m$ that do not vanish simultaneously. If there is a regular map $\psi: Y\longrightarrow X$ such that $\varphi\circ\psi ={\rm Id}_Y$ and $\psi\circ \varphi ={\rm Id}_X$, we say that $X$ and $Y$ are \textit{isomorphic}, and we denote it by $X\cong Y$. If $U\subseteq X$ is a Zariski dense open set and $\varphi:U\longrightarrow Y$ is a regular map, we say that $\varphi$ is a \textit{rational map} from $X$ to $Y$, and denote it by $\varphi: X\dashrightarrow Y$. 
See \cite[Section 1]{breiding2022line} for background on the basic properties of rational maps that are used in this section.

Given a variety $X$, we define its \textit{ideal} as the set
\begin{equation}
    I(X)=\{f\in R \mid f(x)=0 \mbox{ for every } x\in X\}
\end{equation}
of homogeneous polynomials that vanish in every element of $X$. For every ideal $I$, it is possible to find a (not necessarily unique) finite set of polynomials $\{f_1,\ldots,f_k\}\subseteq I$, such that every element $f\in I$ can be written as
\begin{equation}
    f(x)=g_1(x)f_1(x)+\cdots +g_k(x)f_k(x),
\end{equation}
for some polynomials $g_i(x)\in R$. In this scenario, we say that $I$ is \textit{generated} by $\{f_1,\ldots ,f_k\}$, and this is denoted as $I=\langle f_1,\ldots ,f_k\rangle$. 

Given a variety $X$ and its ideal $I(X)=\langle f_1,\ldots ,f_k\rangle$, we say that a point $a\in X$ is \textit{smooth} if the rank of the Jacobian matrix $J(a)\coloneqq \left[\frac{\partial f_i}{\partial x_i} (a)\right]$ is equal to the codimension of $X$. This definition is independent of the choice of generators of $I(X)$. For a broader description and results on the smoothness of algebraic varieties, we refer the reader to~\cite{Gathmann}.  

We use the notation $\vee$ to denote the \textit{join} of two vectors spaces, meaning $U\vee V=\mathrm{span}\{U,V\}$. Similarly, $\wedge$ denotes the intersection of linear spaces.


\section{Anchored Multiview Varieties} \label{s: AMV App}

For a camera matrix $C:\PP^3\to\PP^2$, the \textit{back-projected line} of $x\in \PP^2$ is the line in $\PP^3$ that contains all points that are by $C$ projected onto $x$. Similarly, for an image line $\ell\in \Gr(1,\PP^2)$, its back-projected plane is the plane in $\PP^3$ containing all lines that are by $C$ projected onto $\ell$. Under the identification $\Gr(1,\PP^2)\cong (\PP^2)^\vee\cong \PP^2$ we describe the \textit{back-projected plane} of $\ell$ by its defining linear equation $C^T\ell$. We may parameterize lines in ${\rm Gr}(1,\PP^3)$ by two points spanning it.

Throughout this work, we assume that any camera arrangement has at least one camera and all centers are pairwise disjoint. 

\subsection{The linear isomorphisms}

Consider an arrangement $\widetilde{\mathcal{C}}$ of full-rank $2\times 2$ matrices, and an arrangement $\widehat{\mathcal{C}}$ of full-rank $2\times 3$ matrices. We define $\mathcal{M}_{\widetilde{\mathcal{C}}}^{1,1}$, and $\mathcal{M}_{\widehat{\mathcal{C}}}^{2,1}$, respectively as the Zariski closure of the image of the joint maps
    \begin{equation}
        \begin{split}
            \Phi_{\widetilde{\mathcal C}}: \PP^1&\longrightarrow (\PP^1)^m, \\
            X & \longmapsto (\widetilde C_1X,\ldots,\widetilde C_mX)
        \end{split}
    \end{equation}
and 
    \begin{equation}
        \begin{split}
            \Phi_{\widehat{\mathcal C}}:\PP^2&\dashrightarrow (\PP^1)^m,\\
            X& \longmapsto (\widehat{C}_1X,\ldots \widehat{C}_mX).
        \end{split} 
    \end{equation}

\begin{customtheorem}{\ref{thm: anchoredmulti-isomorphism}}  $ $ 

\begin{enumerate}
     \item Let $\phi_L:L \to \PP^1$ and $\psi_{\mathcal C,i}:C_i\cdot L\to \PP^1$ be any choices of linear isomorphisms. Let $\widetilde{\mathcal{C}}$ denote the arrangement of matrices $\widetilde{C}_i:=\psi_{\mathcal{C},i} \circ C_i\circ \phi_L^{-1} $. Then 
    \begin{align}
        \psi_{\mathcal{C},L}:=(\psi_{\mathcal C,1},\ldots,\psi_{\mathcal C,m}):\mathcal{M}_\mathcal{C}^L\to  \mathcal{M}_{\widetilde{\mathcal{C}}}^{1,1}
    \end{align}
    is a linear isomorphism.
    \item Let $\phi_X: \Lambda(X)\to \PP^2$ and $\psi_{\mathcal C,i}: \Lambda(C_i X)\to \PP^1$ be any choices of linear isomorphisms. Let $\widehat{\mathcal{C}}$ denote the arrangement of matrices $\widehat{C}_i:=\psi_{\mathcal C,i} \circ C_i\circ \phi_X^{-1} $. Then 
    \begin{align}
        \psi_{\mathcal{C},X}:=(\psi_{\mathcal C,1},\ldots,\psi_{\mathcal C,m}): \mathcal{L}_\mathcal{C}^X\to \mathcal{M}_{\widehat{\mathcal{C}}}^{2,1}
    \end{align}
    is a linear isomorphism.
\end{enumerate}
\end{customtheorem}
We interpret $C_i\circ \phi_X^{-1}:\PP^2\dashrightarrow\PP^2$ as a matrix as follows. Let $H$ be a plane in $\PP^3$ disjoint from $X$. Then the following map $\phi_X^{-1}:\PP^2\to \Lambda(X)$ is defined by a linear mapping $f_X:\PP^2\to H$ such that $\phi_X^{-1}(Y)=\mathrm{span}\{X,f_X(Y)\}$. As a matrix, $C_i\circ \phi_X^{-1}$ is equal to $[C_iX]_\times C_i f_X$, where 
\begin{align}
    [a]_\times:=\begin{bmatrix}0 & -a_3 & a_2 \\
    a_3 & 0 & -a_1 \\
    -a_2 & a_1 & 0\end{bmatrix}.
\end{align}

We often work with Zariski closures of images of rational maps. By Chevalley's theorem \cite[Theorem 4.19]{michalek2021invitation}, we may equivalently take Euclidean closures. With this in mind, we can use the following lemma.

\begin{lemma}\label{le: ClosIsomo} Let $\psi:\mathcal X \to\mathcal Y$ be an isomorphism and $U\subseteq \mathcal X,V\subseteq \mathcal Y$ sets whose Euclidean closures equals their Zariski closures. If $\psi(U)=V$, then $\psi(\overline{U})=\overline{V}$.
\end{lemma}

\begin{proof} Take a point $v\in \overline{V}\setminus V$. Then there is a sequence $V\ni v^{(n)}\to v$ in Euclidean topology such that $u^{(n)}=\psi^{-1}(v^{(n)})\in U$ converges in Euclidean topology by continuity of $\psi$ to a point $u\in \overline{U}$ for which $\psi(u)=v$. We have shown $\overline{V}\subseteq \psi(\overline{U})$. Similarly we show $\overline{U}\subseteq \psi^{-1}(\overline{V})$ from which it follows that $\psi(\overline{U})\subseteq \overline{V}$.
\end{proof}

\begin{proof}[Proof of \Cref{thm: anchoredmulti-isomorphism}] $ $

$\textit{1}.$ It is worth noting that $\Phi_\mathcal{C}|_L$ is well-defined everywhere, as $L$ does not contain any center. Additionally, $\Phi_{\widetilde{\mathcal C}}$ is defined everywhere. In particular, the images of both maps are Zariski closed.

By construction, $\psi_{\mathcal C,L}(\Phi_{\mathcal C}|L(X))=\Phi_{\widetilde{\mathcal C}}(\phi_L(X))$, which shows that $\psi_{\mathcal C,L}$ is a well-defined map.

Take a point $x\in\mathcal{M}_{\widetilde{\mathcal{C}}}^{1,1}$, then there is a point $X\in \PP^1$ such that $x_i=\widetilde{C}_iX$ for each $i$. Consider $x'\in \mathcal{M}_{\mathcal{C}}^L$, the image of $X'=\phi_L^{-1}(X)$ such that $x_i'=C_iX'$. By construction, $x$ is the image of $x'$ under $\psi_{\mathcal C,L}$, which shows surjectivity.  

For injectivity, assume that $\psi_{\mathcal C,L}(X)=\psi_{\mathcal C,L}(X')$. Then for each $i$, $C_iX=C_iX'$. However, since the line $L$ does not meet any of the centers, the back-projected lines must meet in exactly one point inside $L$, meaning that $X=X'$.

$\textit{2}.$ As we wish to use \Cref{le: ClosIsomo}, we let 
\begin{align}
    \mathcal X=\Lambda(C_1X)\times \cdots\times \Lambda(C_mX),\quad \mathcal Y=(\PP^1)^m. 
\end{align}
Note that $\psi_{\mathcal C,X}:\mathcal X\to \mathcal Y$ is an isomorphism by construction.
Further, let $U$ be the image of $\Upsilon_\mathcal{C}|_{\Lambda(X)}$, and $V=\mathrm{Im}\;\Phi_{\widehat{\mathcal C}}$. One can show $\psi_{\mathcal{C},X}(U)=V$ via similar calculations to \textit{1}.  
\end{proof}

\indent
\subsection{Irreducibility, dimension, and equations} 

In the main body of the paper it was claimed that the anchored multiview varieties under natural conditions equal,

\begin{align}
    \mathcal{X}_\mathcal{C}^L&=\{(x_1,\ldots,x_m)\in \mathcal{M}_{\mathcal{C}}: x_i\in C_i\cdot L\},\label{eq:defAnc1}\\
    \mathcal{Y}_\mathcal{C}^X&=\{(\ell_1,\ldots,\ell_m)\in \mathcal{L}_{\mathcal{C}}: C_iX\in \ell_i\}.\label{eq:defAnc2}
\end{align}

This provides an alternative characterization to the closure of the images of restrictions of $\Phi_\mathcal{C}$ and $\Upsilon_\mathcal{C}$, which is a useful fact that we formalize in the following lemma. 

\begin{customproposition}{\ref{prop: first-eq}} Consider an arrangement of $m$ cameras $\mathcal{C}=(C_1,\ldots,C_m)$, a point $X\in\PP^3$ and a line $L$ in $\mathbb P^3$ satisfying the conditions of~\Cref{def:AMV}. 
    \begin{enumerate}
        \item If there are two different camera centers $c_i$ and $c_j$ such that the span of $\{c_i,c_j,L\}$ is $\PP^3$, then 
        \begin{equation}
            \mathcal{M}_\mathcal{C}^L=\{(x_1,\ldots,x_m)\in \mathcal{M}_{\mathcal{C}}: x_i\in C_i\cdot L\}.
        \end{equation}

        \item If for each camera center $c_i$, the line spanned by $c_i$ and $X$ does not contain any other camera center, then 
        \begin{equation}
            \mathcal{L}_\mathcal{C}^X=\{(\ell_1,\ldots,\ell_m)\in \mathcal{L}_{\mathcal{C}}: C_iX\in \ell_i\}.
        \end{equation}
    \end{enumerate}
\end{customproposition}

\begin{proof} $ $

$\textit{1}.$ We recall the assumption that $L$ contains no center. Then $\Phi_\mathcal{C}|_L$ is defined everywhere and the image of this map is closed. Let $x\in \mathrm{Im}\Phi_\mathcal{C}|_L$. There is an $X\in L$ such that $x=\Phi_\mathcal{C}(X)$. Therefore $x\in \mathcal{M}_\mathcal{C}$ and $x_i\in C_i\cdot L$. Conversely, if $x\in \mathcal{M}_\mathcal{C}$ and $x_i\in C_i\cdot L$, then since the back-projected line of $x_i$ meet $L$ in unique points $X_i\in L$, we just have to argue that $X_i$ are all the same. This is trivial if there is only one camera. If there are two centers $c_i,c_j$ that together with $L$ span $\PP^3$, then the planes $c_i\vee L$ and $c_j\vee L$ meet in exactly the line $L$. Then the back-projected lines of $x_i,x_j$ must meet inside $L$, implying $X_i=X_j$. For any other center $c_k$, we either have that $c_i,c_k$ and $L$ span $\PP^3$ or $c_j,c_k$ and $L$ span $\PP^3$. This either implies $X_i=X_k$ or $X_j=X_k$ by the above. Either way, repeating this process shows that all $X_i$ are equal and $x=\Phi_\mathcal{C}(X)$ for $X=X_i$.  

$\textit{2}.$ For any line $L\in \Lambda(X)$ that does not meet any center, it is clear that $\ell=\Upsilon_\mathcal{C}(L)$ satisfies $\ell\in \mathcal{L}_\mathcal{C}$ and $C_iX\in \ell_i$. Therefore $\mathcal{L}_\mathcal{C}^X\subseteq  \mathcal{Y}_\mathcal{C}^X$. For the other inclusion, we take an element $\ell\in \mathcal{Y}_\mathcal{C}^X$. If the intersection of the back-projected planes $H_i$ of $\ell_i$ contain a line $L$ through $X$ meeting no center, then $\ell=\Upsilon_{\mathcal{C}}|_{\Lambda(X)}(L)$. This especially happens when $H_i$ intersect in a plane. Note that if the intersection contains a line $L$ that doesn't meet $X$, then the intersection contains the plane $X\vee L$. We are left to check what happens if $H_i$ intersect in exactly a line $L$ that meets a center, say $c_i$. By assumption, no other center is contained in this line. Therefore $H_j,j\neq i$ is equal to $c_j\vee L$. Let $L^{(n)}\in \Lambda(X)$ be any sequence of lines in $H_i$ meeting no centers and such that $L^{(n)}\to L$. It is clear that $c_j\vee L^{(n)}\to H_j$ for $j\neq i$ $n\to \infty$ and $c_i\vee L^{(n)}= H_i$ for each $n$. Then $\Upsilon_\mathcal{C}(L^{(n)})\to \ell$, showing $\ell\in \mathcal{L}_\mathcal{C}^X$ and we are done.  \end{proof}

If the assumptions of \Cref{prop: first-eq} do not hold, then the result doesn't either. In the first statement, let $c_1,c_2$ be centers that together with $L$ span a plane $P$. Given any point $X\in P\setminus \{c_1,c_2\}$, the element $x=(C_1 X, C_2X)$ satisfies that $x\in \mathcal{M}_\mathcal{C}$ and $x_i\in C_i\cdot L$. However, generally for a point $X\in P\setminus L$, we have $x\not\in \mathrm{Im}\Phi_\mathcal{C}|_L$. For the second statement, consider two centers $c_1,c_2$ that together with $X$ span a line $L$. Consider two distinct planes $H_1, H_2$, both containing the line $L$. They meet therefore exactly in $L$ and the pair of corresponding image lines $(\ell_1,\ell_2)$ lies in $ \mathcal{Y}_\mathcal{C}^X$ for the camera arrangement given by these two cameras. However, in the image of $\Upsilon_\mathcal{C}|_{\Lambda(X)}$, the back-projected planes $H_1$ and $H_2$ are always the same.

\begin{customproposition}{\ref{prop: dim}} $\mathcal{M}_\mathcal{C}^L$ and $\mathcal{L}_\mathcal{C}^X$ are irreducible. Further, 
\begin{enumerate}
    \item $\mathcal{M}_\mathcal{C}^L$ is isomorphic to $\PP^1$. In particular, $\dim\mathcal{M}_\mathcal{C}^L=1$.
    \item If the span of the centers $c_i$ and the point $X$ are not collinear, then $\dim \mathcal{L}_\mathcal{C}^X=2$. 
\end{enumerate}
\end{customproposition}

\begin{proof} Both varieties are irreducible since the image of any rational map from an irreducible variety is irreducible.

$\textit{1}.$ Since we assume no center lies in $L$, $\Phi_\mathcal{C}$ restricted to $L$ is defined everywhere, and therefore the image of this restriction equals $\mathcal{M}_\mathcal{C}^L$. This map is further injective since if $x\in \mathcal{M}_\mathcal{C}^L$, then the back-projected line of $x_i\in \PP^2$ meets $L$ in exactly a point $X$, which implies that $X$ is the only point on $L$ for which $x=\Phi_\mathcal{C}(X)$. 

$\textit{2}.$ Note that since $\dim \Lambda(X)=2$, we have $ \dim\mathcal{L}_\mathcal{C}^X\le 2$. Let $U\subseteq \Lambda(X)$ be the subset of lines that meets no center. Without restriction, assume $c_1,c_2$, and $X$ span a plane. Each line $L\in U$ uniquely defines two planes via $c_1\vee L,c_2\vee L$. Since $\dim \Lambda(X)=2$, projection of $ \mathcal{L}_\mathcal{C}^X$ onto the factors of $c_1,c_2$ is at least two dimensional, showing the other inequality $ \dim\mathcal{L}_\mathcal{C}^X\ge 2$. 
\end{proof}

For the result below, let $F^{ij}$ denote the fundamental matrix of $C_i$ and $C_j$, see \cite{Hartley2004,trager2015joint}.

\begin{customproposition}{\ref{prop: eq}}  For a point $X\in \PP^3$ and line $L$ in $\PP^3$, let $\mathcal{C}$ be a generic (random) camera arrangement of $m$ cameras. 
\begin{enumerate}
    \item $x\in \mathcal{M}_\mathcal{C}^L$ if and only if $x_1^TF^{1j}x_j=0$ for every $j=2,\ldots,m$ and $x_i^TC_i\cdot L=0$ for every $i=1,\ldots,m$.
    \item $\ell\in \mathcal{L}_\mathcal{C}^X$ if and only if
    \begin{align}
    \begin{split}
        &\det \begin{bmatrix}
    C_1^T \ell_1 & C_2^T \ell_2 & C_i^T\ell_i 
    \end{bmatrix}=0,\\  
    &\det \begin{bmatrix}
    C_1^T \ell_1 & C_3^T \ell_3 & C_i^T\ell_i 
    \end{bmatrix}=0
    \end{split}
    \end{align}
    for $i=3,\ldots,m$ and $\ell_i^T C_iX=0$ for every $i=1,\ldots,m$.
\end{enumerate}
\end{customproposition}

\begin{proof} Note that in the generic case, the conditions of \Cref{prop: first-eq} hold. 

$\textit{1}.$ Recall that $x_i^TC_i\cdot L=0$ is equivalent to $x_i\in C_i\cdot L$. As in the proof of \Cref{prop: dim}, $x\in \mathcal{M}_\mathcal{C}^L$ is uniquely determined by the intersection $X\in L$ of its back-projected lines. The back-projected lines $L_i$ of $x_i$ intersect if and only if the pairs $(L_1,L_i)$ intersect for $i\ge 2$, which in turn is equivalent to $x_1^TF^{1i}x_i=0$ for the fundamental matrix $F^{1i}$. 
 
$\textit{2}.$ Recall that $\ell_i^T C_iX=0$ is equivalent to $C_iX\in \ell_i$. By \cite[Theorem 2.5]{breiding2022line}, $\ell\in \mathcal L_{\mathcal C}$ if and only if the back-projected planes meet in a line (assuming generic centers), and  
 \begin{align} \label{eq: ijk}
        \det \begin{bmatrix}
    C_i^T \ell_i & C_j^T \ell_j & C_k^T\ell_k 
    \end{bmatrix}=0,
\end{align}
is equivalent to the back-projected planes of $\ell_i,\ell_j,\ell_k$ meeting in at least a line. By \Cref{prop: first-eq}, we are left to show direction $\Leftarrow$. Let $H_i$ denote the back-projected plane of $\ell_i$. For $m=2$, the two back-projected planes always meet. For $m\ge 3$ we have that $c_1,c_2,c_3$ with $X$ span $\PP^3$ by genericity. Especially, the back-projected planes of $\ell_1,\ell_2,\ell_3$ meet in a line by setting $i=1,j=2,k=3$ in \Cref{eq: ijk}. Also, since $c_1,c_2,c_3, X$ span $\PP^3$, they meet exactly in a line. If $m\ge 4$, it suffices to show that $H_l$ for $l\ge 4$ meets $H_1,H_2,H_3$ in a line. Note that either $H_1, H_2$ or $H_1, H_3$ meet exactly in a line. Let $i,j\in \{1,2,3\}$ denote indices for which this happens. Then for $i,j$ and $k=4$, \Cref{eq: ijk} guarantees that $H_i,H_j,H_4$ meet in exactly a line, which suffices. 
\end{proof}

\subsection{Smoothness and multidegrees} First, similar to what is done in the proof of the smoothness properties of the multiview variety $\mathcal{M}_\mathcal{C}$ in \cite{trager2015joint}, we use that multiview varieties are isomorphic to corresponding varieties of back-projected lines or planes.  

\begin{customproposition}{\ref{prop: smooth}}  $ $

\begin{enumerate}
    \item $\mathcal{M}_\mathcal{C}^L$ smooth.
    \item If there are exactly two cameras, or the centers together with the point $X$ span $\PP^3$, then $\mathcal{L}_\mathcal{C}^X$ is smooth. 
\end{enumerate}
\end{customproposition}

\begin{proof} $ $

$\textit{1}.$ Since $\mathcal{M}_\mathcal{C}^L$ is isomorphic to $\PP^1$ by \Cref{prop: dim}, it is smooth.

$\textit{2}.$ Assume that the line $c_i\vee X$ contains $c_j$ for $j\neq i$. In the image we always have $H_i=H_j$ for the back-projected planes of $\ell_i,\ell_j$. Therefore $\mathcal{L}_\mathcal{C}^X$ is isomorphic to $\mathcal{L}_{\mathcal{C}'}^X$, where $\mathcal{C}'$ is equal to $\mathcal{C}$ after having removed the smallest amount of cameras from $\mathcal{C}$ such that each line $c_i\vee X$ contains exactly one center, namely $c_i$ itself. We, therefore, assume now that $\mathcal{C}$ has this property: $c_i\vee X$ contains only the center $c_i$ for each $i$.

If $m=1$, then one can check that $\mathcal{L}_{\mathcal{C}}^X$ is isomorphic to $\PP^1$ and if $m=2$, that $\mathcal{L}_{\mathcal{C}}^X$ is isomorphic to $\PP^1\times \PP^1$. The latter is for instance because any choice of $\ell_1,\ell_2$, where $C_iX\in \ell_i$ guarantees that the back-projected planes $H_i$ meet in a line containing $X$. Therefore we now assume that there are at least three cameras. 

First, we note that $\Lambda(X)$ is a smooth variety and the lines $c_i\vee X$ are smooth subvarieties. Up to linear transformation, we may assume that $X=(1:0:0:0)$ without loss of generality. Let $a=(a_0:a_1:a_2:a_3)$ be distinct from $X$. Then $(0:0:a_0:0:a_1:a_2)$ are the Plücker coordinates of $a\vee X$. In particular, 
\begin{align}
    \Lambda(X)=\{w\in \PP^5: w_0=w_1=w_3=0\}.
\end{align}
In Plücker coordinates, the line $L=a\vee X$ in coordinates $w$ and the fixed point $b=(b_0:b_1:b_2:b_3)\in \PP^3$ spann the plane:
\begin{align}\label{eq: PlüPlane}
    (0: w_2b_1-w_4b_0: w_2b_2-w_5b_0:w_4b_2-w_5b_1).
\end{align}
The three linear non-zero functions in $w$ in \Cref{eq: PlüPlane} vanishes if and only if $b$ lies in the line $L$. Denote them by $f_{b,1},f_{b,2},f_{b,3}$ and $f_b(L)=(0:f_{b,1}(L):f_{b,2}(L):f_{b,3}(L))$ for a line $L\in \Lambda(X)$. Let $c\neq X$. Since the blow-up of a linear space at a linear space is smooth, then
\begin{align}
   \Gamma_{C}:= \overline{\{(L,f_c(L)):c\vee X\neq L\in \Lambda(X) \}},
\end{align}
is a smooth variety in $\Lambda(X)\times \Gr(1,\PP^3)$. Keep in mind that $f_c(L)=c\vee L$. Next we consider the joint blow-up $\Gamma_\mathcal{C}$ defined as
\begin{align}
\overline{\{(L,f_{c_1}(L),\ldots,f_{c_m}(L)): c_i\vee X\neq L\in \Lambda(X)\}}
\end{align}
in $\Lambda(X)\times \mathrm{Gr}(1,\PP^3)^m$. Take an element $\ell\in \mathcal{L}_\mathcal{C}^X$. By the assumption that there are at least three cameras, and the centers together with the point $X$ span $\PP^3$, we have that the back-projected planes meet in exactly a line $L$ containing $X$. We have also assumed that $L$ contains at most one center. If $c_i\in L$, then fix the index $i$, otherwise choose any index $i$. Consider the natural projection,
\begin{align}
\pi_i:\Gamma_\mathcal{C}\to \Gamma_{C_i}.
\end{align}
Restricting to the set where $L$ meets none of the other centers $c_j,j\neq i$, this map is an isomorphism. Since $\Gamma_{c_i}$ is smooth, that means that any element of $\Gamma_\mathcal{C}$, where $L$ does not meet $c_j,j\neq i$ is smooth. But since $i$ was arbitrary, all of $\Gamma_\mathcal{C}$ is smooth. Finally, since the back-projected planes always meet in exactly a line, the projection onto the last $m$ coordinates
\begin{align}
    \pi:\Gamma_{\mathcal{C}}\to \widetilde{\mathcal{L}}_\mathcal{C}^X,
\end{align}
is an isomorphism and therefore $ \widetilde{\mathcal{L}}_\mathcal{C}^X$ is smooth, but this is the variety of the back-projected planes of $\mathcal{L}_\mathcal{C}^X$. In particular, they are isomorphic, and therefore $\mathcal{L}_\mathcal{C}^X$ is also smooth.
\end{proof}

We denote by~$L_{d}\subseteq \mathbb P^{h}$ a general linear subspace of codimension $d$, meaning dimension $h-d$. The \textit{multidegree} of a variety $\mathcal{X}\subseteq \PP^{h_1}\times \cdots \times\PP^{h_m}$ is the function 
\begin{align}
    D(d_{1},\dots,d_{m}):= \#( \mathcal{X}\cap (L_{d_{1}}^{(1)}\times \cdots\times L_{d_{m}}^{(m)})), 
\end{align}
for $(d_{1},\dots,d_{m})\in \mathbb N^{n}$ such that $d_{1}+\cdots +d_{m} = \operatorname{dim} \mathcal{X}$. First note that for any multiview variety, the function $D$ is symmetric under generic camera conditions. This implies that for any permutation $\sigma \in S_n$, $D(d_{1},\dots,d_{m})$ is equal to $D(d_{\sigma{(1)}},\dots,d_{\sigma{(m)}})$.

\begin{proposition} Let $\mathcal{C}$ be a generic arrangement of cameras.

\begin{enumerate}
    \item The multidegree of $\mathcal{M}_\mathcal{C}^L$ is given by the single number $D(1,0,\ldots,0)=1$.
    \item The multidegree of $\mathcal{L}_\mathcal{C}^X$ is given by the two numbers $D(2,0,\ldots,0)=0$ and $D(1,1,0,\ldots,0)=1$.
    \end{enumerate}
\end{proposition}

\begin{proof} $ $

$\textit{1}.$ Since $\mathcal{M}_{\mathcal{C}}^L$ is of dimension 1 and due to symmetry, we only need to consider one number, namely $D(1,0,\ldots,0)$. A generic linear form intersecting the line $C_1\cdot L$ leaves one point, say $x_1$. Its back-projected line meets $L$ in a unique point $X$. Recall that any point outside the back-projected line is not projected onto $x_1$ by $C_1$. Since $\mathcal{M}_\mathcal{C}^L$ equals the image of $\Phi_{\mathcal{C}}|_L$, $X$ is therefore the unique point on $L$ such that $x=\Phi_\mathcal{L}(X)$.

$\textit{2}.$ By symmetry and the fact that $\dim \mathcal{L}_\mathcal{C}^X=2$, we only need to determine $D(2,0,\ldots,0)$ and $D(1,1,0,\ldots,0)$. Two generic linear forms intersecting $\Lambda(C_1X)\subseteq \PP^2$ is an empty set, which is why $D(2,0,\ldots,0)=0$. Intersecting $\Lambda(C_1X)$ and $\Lambda(C_2X)$ each with generic linear forms leaves one point in each copy of $\PP^2$, say $\ell_1$ and $\ell_2$ that intersect $C_1X$ and $C_2X$ respectively. Since they are generic such that $C_iX\in \ell_i$ their back-projected planes $H_i$ both contain $X$. By genericity, $H_i$ meet in a unique line through $X$ that meets no center, showing $D(1,1,0,\ldots,0)=1$.
 \end{proof}


\subsection{The Euclidean distance problem} This section will be used for the proof of \Cref{cor: EDD-p}. It also explains in more detail the reduction of parameters mentioned in \Cref{sss: Red}, via \Cref{thm: EDbij}. 

It is not always true that linearly isomorphic varieties have the same Euclidean distance degree. Take for instance the circle and ellipse in $\RR^2$. The circle has EDD $2$ and the ellipse has EDD $4$. However, under additional assumptions, the EDD is the same:

\begin{proposition}\label{prop: EDbij} Let $X\subseteq \CC^n, Y\subseteq \CC^m$ (with $n\ge m$) and let $\psi:X\to Y$ sending $x$ to $Ax+b$ be an affine isomorphism given by a real full-rank matrix $A$ such that $AA^T=I$. Fix a generic $u\in \CC^n$. Then, $\mathrm{EDD}(X)=\mathrm{EDD}(Y)$. 

In particular, if $y_1^*,\ldots,y_k^*$ are the solutions to the critical equations of the ED problem on $Y$ given $Au+b$, then $x_1^*=A^T(y_1^*-b),\ldots,x_k^*=A^T(y_k^*-b)$ are the solutions to the critical equations of the ED problem on $X$ given $u$. 
\end{proposition}

We work with the critical equations, as defined in \cite{draisma2016euclidean}. Before the proof, we recall some linear algebra: We use that $A^TA$ is a projection matrix onto $\mathrm{Im}A^TA$, and that $\CC^n=\mathrm{Im} A^TA \oplus \ker A^TA$, by which we mean that any $x\in \CC^n$ can be written as a unique sum $x_1+x_2$ with $x_1\in \mathrm{Im} A^TA,x_2\in \ker A^TA$. Over the real numbers, this is an orthogonal decomposition, i.e. for $x_1,x_2$ as above we have $x_1\cdot x_2=0$ with respect to the standard inner product. Moreover, the rank of $A^TA$ is the rank of $A$. This implies that $A^T$ is injective on $\mathrm{Im} A$ and $X\subseteq \mathrm{Im}A^TA$.

\begin{proof} It is not hard to see that shifting a variety by a constant does not change the EDD. Therefore, we put $b=0$ and continue. 

Let $I_X=\langle f_1,\ldots,f_r\rangle$ be the defining ideal of $X$. Let $g_i=f_i\circ A^T$. We claim that $I_Y=\langle g_1,\ldots,g_r\rangle$ is the defining ideal of $Y$. Indeed, $y\in Y$ if and only if $A^Ty\in X$ if and only if $f_i(A^Ty)=0$ for each $i$. Further, a full-rank linear change of coordinates preserves the radicality of ideals. 

Since $\psi$ is an isomorphism, $x\in X$ is smooth if and only if $Ax\in Y$ is smooth. Now for a generic $u\in \CC^n$, let $z^*=(z_1,\ldots,z_m)\in Y$ be smooth and a solution to the critical equations given $Au$. Write $c_Y=\mathrm{codim}_{\CC^m}Y$. Then
\begin{align}
     g_i(z^*)=0~\mbox{ for all } i~\&~\mathrm{rank} \begin{bmatrix} (z^*-Au)^T\\
    \nabla g_1(z^*)\\
     \vdots \\
     \nabla g_k(z^*)  
    \end{bmatrix}=c_Y.
\end{align}
We define $w^*=A^Tz^*$ and prove that its a solution to the critical equations of $X$ given $u$. First, note that $f_i(w^*)=0$ for each $i$ by construction, and $A(w^*-A^TAu)=z^*-Au$. By the chain rule, $\nabla g_i(z^*)=\nabla f_i (A^Tz^*) A^T$. Thus we also have
\begin{align} \label{eq: rankcoY}
    \mathrm{rank} \left(\begin{bmatrix} (w^*-A^TAu)^T\\
    \nabla f_1(w^*)\\
     \vdots \\
     \nabla f_k(w^*)  
    \end{bmatrix}A^T\right)=c_Y.
\end{align}
Note that the submatrix of the last $k$ rows of the matrix in \Cref{eq: rankcoY} has rank $c_Y$. Now we argue that for $c_X=\mathrm{codim}_{\CC^n}X$,
\begin{align}\label{eq: rankcoX}
    \mathrm{rank} \begin{bmatrix} (w^*-A^TAu)^T\\
    \nabla f_1(w^*)\\
     \vdots \\
     \nabla f_k(w^*)  
    \end{bmatrix}=c_X.
\end{align}
Because $w^*$ is smooth in $X$, last $k$ rows of the matrix in \Cref{eq: rankcoX} are of rank $c_X$. The $(w^*-A^TAu)^T$ lies in the row span of those $k$ rows, and observe that $w^*-A^TAu$ lies in $ \mathrm{Im} A^TA$. This is because $z^*$ lies in the image of $A$. Therefore, 
\begin{align}
    A(w^*-A^TAu)\in \mathrm{span}\{A\nabla f_i(w^*)^T\}
\end{align}
implies 
\begin{align}
     w^*-A^TAu\in \mathrm{span}\{A^TA\nabla f_i(w^*)^T\}.
\end{align}
Since $X\subseteq \mathrm{Im}A^TA$, it follows that $\nabla f_i(w^*)^T$ span $\ker A^TA$. This is because $f_i$ generate the (real) linear forms $l_j$ that vanish on this linear space $\mathrm{Im}A^TA$ and their gradients span the (real) orthogonal complement $\ker A^TA$. So let $\lambda_i$ be such that $w^*-A^TAu$ equals the sum of $\lambda_iA^TA\nabla f_i(w^*)^T$. Then $w^*-A^TAu$ equals $\sum \lambda_i\nabla f_i(w^*)^T-v$, for some $v\in \ker A^TA$, spanned by $ \nabla f_i(w^*)^T$. This proves \Cref{eq: rankcoX}.  

Finally, we motivate why we can change $A^TAu$ to $u$ in \Cref{eq: rankcoX}. Showing that $A^TAu-u$ is linearly dependent on $\nabla f_i(w^*)^T$ is suffcient due to the fact that $(w^*-A^TAu)+(A^TAu-u)=w^*-u$. However, $A^TAu-u$ lies in $\ker A^TA$, and therefore this follows from the above. 

For the other direction, let $w^*$ be a smooth point satisfying
\begin{align} \label{eq: fullrankcoX}
     f_i(w^*)=0~\mbox{ for all } i~\&~\mathrm{rank} \begin{bmatrix} (w^*-u)^T\\
    \nabla f_1(w^*)\\
     \vdots \\
     \nabla f_k(w^*)  
    \end{bmatrix}=c_X.
\end{align}
Then $(w^*-u)^T$ is a linear combination of the rows $\nabla f_i(w^*)$. Then $(w^*-u)^TA^T$ is a linear combination of $\nabla f_i(w^*)A^T$. Writing $z^*=Aw^*$ and recalling that this is a smooth point of $Y$, we have that $(z^*-Au)^T$ is a linear combination of $\nabla g_i(z^*)$. Therefore, 
\begin{align}\label{eq: fullrankcoY}
    g_i(z^*)=0~\mbox{ for all } i~\&~ \mathrm{rank} \begin{bmatrix} (z^*-Au)^T\\
    \nabla g_1(z^*)\\
     \vdots \\
     \nabla g_k(z^*)  
    \end{bmatrix}=c_Y,
\end{align}
are all satisfied.
\end{proof}

In the theorem below we use the relation between $\mathcal{C}$ and $\widetilde{\mathcal{C}}$ and $\widehat{\mathcal{C}}$ from \Cref{thm: anchoredmulti-isomorphism}.

\begin{theorem}\label{thm: EDbij}  Let $U_i\subseteq \PP^{2}$ be affine patches and write $U=U_1\times \cdots\times U_m$. Fix real matrices $A_i:\CC^3\to \CC^2$ such that $A_iA_i^T=I$. Let $A:(\CC^3)^m\to (\CC^2)^m$ be the map that sends $(x_1,\ldots,x_m)\in (\PP^2)^m$ to $(A_1x_1,\ldots,A_mx_m)\in (\PP^1)^m$. Let $u\in U_1\times \cdots\times U_m$ be generic.
\begin{enumerate}
    \item Assume $U_i\cap (C_i\cdot L)\neq \emptyset$ for each $i$. Write $V_i=A_i(U_i\cap (C_i\cdot L))\subseteq \RR^2$, and let $V=V_1\times\cdots\times V_m$. If $y^*$ is a critical point of the ED problem for $\mathcal{M}_{\widetilde{\mathcal{C}}}^{1,1}\cap V$ given $Au$, then $x^*=A^Ty^*$ is a critical point of the ED problem for $\mathcal{M}_\mathcal{C}^L\cap U$ given $u$.
    \item Assume $U_i\cap \Lambda(C_iX)\neq \emptyset$ for each $i$. Write $V_i=A_i(U_i\cap \Lambda(C_iX))$, and let $V=V_1\times\cdots\times V_m$. If $y^*$ is a critical point of the ED problem for $\mathcal{M}_{\widehat{\mathcal{C}}}^{2,1}\cap V$ given $Au$, then $x^*=A^Ty^*$ is a critical point of the ED problem for $\mathcal{L}_\mathcal{C}^X\cap U$ given $u$. 
\end{enumerate}
In both cases, this is a bijection of critical points.
\end{theorem}

\begin{proof} It is a consequence of \Cref{thm: anchoredmulti-isomorphism} that $A$ is an isomorphism of affine varieties in both \textit{1}. and \textit{2} (we set $\psi_{\mathcal{C},i}=A_i$). Then we can directly apply \Cref{prop: EDbij}.
\end{proof}


\section{Euclidean Distance Degree Preliminaries}\label{s: EDDprelim}

The main theorem of this article is:

\begin{customtheorem}{\ref{thm: EDDthm}} Let $\mathcal{C}$ be a generic arrangement of $m$ cameras.  
\begin{enumerate}
    \item $\mathrm{EDD}( \mathcal{M}_\mathcal{C}^L)=3m-2 $.
    \item If $m\ge 3$, then $\mathrm{EDD}( \mathcal{L}_\mathcal{C}^X)=\frac{9}{2}m^2-\frac{19}{2}m+3.$
\end{enumerate}
\end{customtheorem}

In order to compute these two Euclidean distance degrees we make use of the following theorem:

\begin{theorem}[Theorem 3.8 of \cite{EDDegree_point}]
Let $X\subseteq \CC^n$ be a smooth variety and let $U_{\beta}$ denote the complement of the hypersurface $\sum_{1\leq i\leq n}(z_{i}-\beta_{i})^2+\beta_{0}=0$ in $\mathbb C^n$ where $z\in\mathbb C^n$ and $\beta\in\mathbb C^{n+1}$. Then,
\begin{align}
    \mathrm{EDD}(X)=(-1)^{\operatorname{dim} X}\chi(X\cap U_{\beta}).
\end{align}
\end{theorem}

Here $\chi$ is the topological Euler characteristic. In the next section, we closely follow \cite{EDDegree_point}, by specializing their techniques to our setting. First, we provide the reader with helpful preliminaries. We often take this section for granted and do not always refer to specific results from it.

We have verified with numerical evidence that these formulas hold for $m\le 10$. The code is attached.


\subsection{The Euler characteristic} 

There are different approaches to defining the Euler characteristic of a topological space. References to the broader topic of algebraic topology include \cite{may1999concise,hatcher2005algebraic}. For instance, given a \textit{triangulation} of a topological space, the Euler characteristic is the alternating sum
\begin{align}
    k_0-k_1+k_2-\ldots,
\end{align}
where $k_i$ is the number of simplices of dimension $i$. An \textit{$n$-simplex} is a polytope of dimension $n$ with $n+1$ vertices, and a \textit{triangulation} is essentially a way of writing a space as a union of simplices that intersect in a good way. Importantly, all real and complex algebraic varieties can be triangulated \cite{hofmann2009triangulation} with respect to Euclidean topology.

The Euler characteristic can more generally be defined for CW complexes and any topological space through singular homology. For spaces where all definitions apply, they are the same.

The following is used in \cite{EDDegree_point}. 
\begin{lemma}\label{le: Eul1} Let $N,M$ be subvarieties of a complex variety. 
\begin{enumerate}
    \item $\chi(M\cup N)=\chi(M)+\chi(N)-\chi(M\cap N).$
\item $
    \chi(M\setminus N)=\chi(M)-\chi(N).$
\end{enumerate}

\end{lemma}
The above does not hold over the real numbers. For instance, $\chi(\RR)=1$, while $\chi(\{x\})=1$ and $\chi(\RR\setminus \{x\})=2$.

\begin{lemma}[{\hspace{1sp}\cite[Section 2.1]{hatcher2005algebraic}}]\label{le: Eul2} Let $f: X\to Y$ be a homeomorphism, such as an isomorphism between varieties, then 
\begin{align}
    \chi(X)=\chi(Y).
\end{align}
\end{lemma}

\begin{lemma}[{\hspace{1sp}\cite[Chapter 10, Section 1]{may1999concise}}]\label{le: Eul3} The Euler Characteristic of $\PP^n$ is $n+1$.  
\end{lemma}


\subsection{Chow rings}\label{ss: Chow}

We refer to \cite{fulton2013intersection,eisenbud-harris:16} for a thorough treatment of intersection theory, and \cite{Fulvio21notes} for a friendly introduction. Here we recall the basic definitions and results that are needed to understand this material.

Let $X$ be a variety. We denote by $Z(X)$ the
free abelian group of formal integral linear combinations of irreducible subvarieties of $X$. An \textit{effective cycle} is a formal sum $\sum n_i Y_i$ of irreducible subvarieties $Y_i$ with $n_i\ge0$. A \textit{zero-cycle} is a formal sum of zero-dimensional varieties $Y_i$. The \textit{degree} of a zero-cycle is the sum of the associated integers $n_i$ as in \cite[Definition 1.4]{fulton2013intersection}. We say that two irreducible subvarieties $Y_0, Y_\infty \in Z( X)$ are \textit{rationally equivalent}, and write $Y_0\sim Y_\infty$ or $Y_0\equiv Y_\infty$ if there exists
an irreducible variety $W \subseteq   X \times\PP^1$, whose projection onto $\PP^1$ is dense, such that $W \cap ( X \times  \{(1:0)\}) = Y_0$ and $W \cap  (X \times \{(0:1)\}) = Y_\infty$. 

The Chow group of $X$ is
\begin{align}
    \mathrm{CH}( X) = Z( X)/\sim.
\end{align}
For a subvariety, $Y \subseteq   X$, write $[Y]$ for the class in $\mathrm{CH}( X)$ of its associated effective cycle. We now aim to turn this group into a ring, by giving it a multiplicative structure.

Let $X$ be an irreducible variety and let $Y_1, Y_2$ be subvarieties. $Y_1$ and $Y_2$ \textit{intersect transversely} at $p \in Y_1 \cap Y_2$ if $Y_1, Y_2$ and $X$ are smooth at $p$ and
$T_pY_1 + T_pY_2 = T_pX$. Further, $Y_1$ and $Y_2$ are \textit{generically transverse} if they intersect transversely at generic points of every irreducible component of the intersection $Y_1 \cap Y_2$.

\begin{theorem}\label{thm: prod} Let $ X$ be a smooth variety. Then there is a unique product structure on $\mathrm{CH}( X)$
such that whenever $A, B$ are generically transverse subvarieties of $X$, then $[A][B] = [A \cap B]$.
This product makes $\mathrm{CH}( X)$ into a graded ring, where the grading is given by codimension.
\end{theorem}

A natural example of Chow rings are those of products of projective space,
\begin{align}
    \mathrm{CH}((\PP^n)^s)\cong \ZZ[[H_1],\ldots,[H_s]]/\langle [H_1]^{n+1},\ldots,[H_s]^{n+1}\rangle.
\end{align}
In the above ring isomorphism, $[H_i]$ represent the class of a hyperplane in $\mathrm{CH}(\PP^n)$ in factor $i$.

To a morphism of smooth varieties $f:X\to Y$,~we can associate, the \textit{pushforward} $f_*:\mathrm{CH}(X)\to \mathrm{CH}(Y)$ and the \textit{pullback}~$f^*:\mathrm{CH}(Y)\to \mathrm{CH}(X),$ two Chow ring maps. 

We define the pushforward on irreducible subvarieties $A\subseteq X$ by setting 
\begin{align}
    f_*(A):=\begin{cases}0 & \textnormal{if the generic fiber of }f|_A \\
    &\textnormal{is infinite},\\
    d[f(A)] & \textnormal{if the generic fiber of }f|_A\\
    &\textnormal{has cardinality }d.\end{cases}
\end{align}
By \textit{generic fiber} we mean $f|_A^{-1}(y)$ for generic $y\in f(A)$.

We say that $A\subseteq Y$ is \textit{generically transverse} to $f$ if $f^{-1}(A)$ is generically reduced and the codimension of $f^{-1}(A)$ in $X$ equals the codimension of $A$ in $Y$. The pullback $f_*$ is defined as the unique map $\mathrm{CH}(Y) \to \mathrm{CH}(X)$ such that, if $A \subseteq Y$ is generically transverse to $f$, then $f^*[A] := [f^{-1}(A)]$; see \cite[Theorem 1.23]{eisenbud-harris:16}.


\subsection{Chern classes}\label{ss: Chern} In intersection theory, Chern classes are algebraic invariants of a variety that lie in its Chow ring. General references again include \cite{fulton2013intersection,eisenbud-harris:16}. Here we only state the properties of them that we use.

Chern classes $c(E)$ are in general defined for vector bundles $E$, but when the vector bundle is the tangent bundle of a smooth variety $X$, then we write $c(X)$ for the total Chern class.

\begin{lemma}[Whitney Sum Formula {\cite[Theorem 3.2]{fulton2013intersection}}] For a short exact sequence $0\to E'\to E\to E''\to 0$ of vector bundles on a variety $X$, we have for the total Chern classes that
\begin{align}
    c(E)=c(E')c(E'').
\end{align}
\end{lemma}

By a \textit{divisor} of $X$ we mean a subvariety of codimension one. Let $i: A\hookrightarrow X$ be the inclusion map for a subvariety $A\subseteq X$. For an element $[V]\in \mathrm{CH}(X)$, the restriction $[V]|_A$ denotes the pullback $i^*[V]$. If $V$ is generically transverse to $i$, then $[V]|_A=[V\cap A]$. We observe that $[V]|_A[U]|_A$ equals $i^*[V]i^*[U]$ and since $i^*$ is a ring homorphism, this equals $i^*([V][U])=([V][U])_A$.

\begin{lemma}[Adjunction Formula {\cite[Example 3.2.11]{fulton2013intersection},\cite[Theorem 5.3]{eisenbud-harris:16}}] If $X$ is smooth variety and $D$ a smooth divisor on $X$, then 
\begin{align}
    c(D)=\frac{c(X)|_D}{(1+[D])|_D}.
\end{align}
\end{lemma}

\begin{lemma}[Functoriality {\cite[Theorem 5.3]{eisenbud-harris:16}}] Let $f:X\to Y$ be morphism of smooth varieties, then 
\begin{align}
    f^*c(E)=c(f^*E),
\end{align}
for vector bundles $E$ on $Y$.
\end{lemma}
By putting $E$ to be the tangent bundle of $Y$, its pullback equals $X$ if $f$ is an isomorphism, which we make precise below.

\begin{lemma} Let $f:X\to Y$ be an isomorphism, then 
\begin{align}
    c(X)=f^*c(Y).
\end{align}
\end{lemma}

\begin{lemma}[{\hspace{1sp}\cite[Example 3.2.11]{fulton2013intersection}}]\label{le: ChernP} Let $[H]$ be the class of a hyperplane of $\PP^n$. Then we have
\begin{align}
    c(\PP^n)=(1+[H])^{n+1}.
\end{align}
\end{lemma}

An important property we use is the next result. The \textit{top} Chern class $c_{\mathrm{top}}(X)$ of $X$ is the zero-cycle part written of $c(X)\in \mathrm{CH}(X)$.

\begin{theorem}[Chern-Gauss-Bonnet {\cite[Section 3.3]{griffiths2014principles}}]\label{thm: CGB} For a smooth variety $X$, we have
\begin{align}
    \chi(X)=\deg(c_{\mathrm{top}}(X)).
\end{align}
\end{theorem}

It happens that authors use the integration symbol for the degree of a zero-cycle $[Z]$ in $X$ the follows sense,
\begin{align}
 \int_{X}[Z]:=\deg(Z).
\end{align}
More generally, let $[Z]$ be a formal sum of irreducible subvarieties of $X$ of codimension $k$. Consider the inclusion $i:A\hookrightarrow X$ for a $k$-dimensional variety $A$. We get the following,
\begin{align}
 \int_{A}[Z]|_A=\int_A i^*[Z]=\int_X [Z][A],
\end{align}
where we assume that $i^{-1}(Z)$ is a $0$-dimensional.

Next, we consider Chern classes of blow-ups. First some notation. Let $X\subseteq Y$ be an inclusion of smooth varieties. Let $\Tilde{Y}$ be the blow-up of $Y$ at $X$. Let $\Tilde{X}$ be the exceptional locus. Let $\pi,\rho$ be the projection maps and $j,i$ the inclusion maps. The following diagram commutes:
\begin{equation}\label{eq: diagram} \begin{tikzcd}
\Tilde{X} \arrow{r}{j} \arrow[swap]{d}{\rho} & \Tilde{Y}  \arrow{d}{\pi} \\%
X \arrow{r}{i}& Y
\end{tikzcd}.
\end{equation}
Porteus' formula \cite[Theorem 15.4]{fulton2013intersection} gives an expression for the Chern class of $\Tilde{Y}$ in terms of the Chern classes of $X$ and $Y$. For our purposes, we only need the following special case of this theorem, which follows from \cite[Example 15.4.2]{fulton2013intersection} as stated in \cite{EDDegree_point}. 

\begin{theorem} In \Cref{eq: diagram}, let $X$ be a set of $m$ distinct points $X_i$. Then 
\begin{align}
c(\Tilde{Y})=\pi^*c(Y)+\sum_{i=1}^m\Big((1+\eta_i)(1-\eta_i)^d-1\Big),
\end{align}
where $d=\dim Y$ and $\eta_i=j_*(\rho^*[X_i])$.
\end{theorem}


\subsection{Linear systems} In the proof of \Cref{thm: EDDthm}, we use the language of linear systems. We don't go through many details here, instead, we recall basic definitions. An introduction to line bundles and other relevant concepts are given in the lecture notes of Vakil \cite{vakil1999introduction}. For a rational function $s$ on a projective variety $X$, we define $(s)=\sum \mathrm{ord}_Z(s) Z\in \mathrm{CH}(X)$, where $\mathrm{ord}_Z(s)$ is the order of $s$ at the point $Z$. Two divisors $D,D'$ are \textit{linearly equivalent} if $D'= D+(s)$ for some rational function $s$.

\begin{definition} Let $X$ be a smooth variety. A complete linear system $|D_0|$ of an effective divisor $D_0$ ($\sum n_iD_i$ so that $n_i\ge 0$) is the set of all effective divisors linearly equivalent to it. A linear system is a linear subspace of a complete linear system.
\end{definition}

\begin{definition} Let $D_0$ be an effective divisor. $\Gamma(X,\mathcal{O}(D_0))$ is the of global sections $s$ on $X$ with $(s)+D_0\ge 0$.
\end{definition}

 $\Gamma(X,\mathcal{O}(D_0))$ is interpreted as a complete linear system via the map $f\mapsto (f)+D_0 \in \mathrm{CH}(X)$. Since $(f)=(g)$ if and only if they differ by a non-zero scalar (zeros and poles determine a rational function), $\Gamma(X,\mathcal{O}(D_0))$ can be viewed as a projective space. Subspaces of $\Gamma(X,\mathcal{O}(D_0))$ are also called linear systems.

The \textit{base locus} of a linear system is the intersection of the zero sets of all global sections on $X$ of the linear system. A linear system is \textit{basepoint free} if the base locus is empty. In other words, for every point  $x\in X$, there is a global section $s$ such that $s(x)\neq 0$.

\begin{lemma}\label{le: basepoint-free} The restriction of a basepoint free linear system is basepoint free.  
\end{lemma}

\begin{proof} Let $V$ be a subvariety of $X$. Take $v\in V$. Since $X$ is basepoint free, there is a global section $s$ of the linear system for $X$ such that $s(x)\neq 0$. Restricting this section to $V$ we get a global section $s|_V$ for the restricted linear system that is non-zero on $v$.
\end{proof}

A linear system on a smooth variety $X$ is called very ample if it allows the variety to be embedded into a projective space in a way that preserves its geometry. For a basepoint free linear system, let $a_0,\ldots,a_n$ be global sections that do not simultaneously vanish. The linear system is \textit{very ample} if
\begin{align}\begin{aligned}
    g:X&\to \PP^n,\\
    x&\mapsto (a_0(x):\cdots:a_n(x))
\end{aligned}
\end{align}
is a \textit{closed immersion}. This means that $g$ is isomorphic onto its image, or that $g^{*}(O(1))$, the pullback of the hyperplane bundle $O(1)$ on $\mathbb P^n$, is isomorphic to $L$. The restriction of a very ample linear system is very ample. 

We define a divisor $D_0$ to be basepoint free if its linear system $\Gamma(X,\mathcal{O}(D_0))$ is basepoint free. We define a divisor to be very ample analogously. 

One importance of basepoint free linear systems comes in the form of this celebrated result: 

\begin{theorem}[Bertini's Theorem {\cite[Section 1.1]{griffiths2014principles}}]\label{thm: Bertini} 
Let $X$ be a smooth complex variety and let $\Gamma$ be a positive dimensional linear system
on $X$. Then the general element of $~\Gamma$ is smooth away
from the base locus.
\end{theorem}


\subsection{Whitney stratification}\label{ss: Whitney} 

A natural way to partition a variety $X$ is via the inclusion
\begin{align}
    X\supseteq \mathrm{sing}(X)\supseteq \mathrm{sing}(\mathrm{sing}(X))\supseteq \cdots,
\end{align}
where $\mathrm{sing}$ denote the singular locus. However, not all points of $\mathrm{sing}(X)\setminus \mathrm{sing}(\mathrm{sing}(X))$ necessarily look locally the same. A more fine grained version of this partition is called a \textit{Whitney stratification} \cite{trotman2020stratification}. We don't recall the definition here, because all we need to know is that a Whitney stratification of a smooth variety $X$ is $\mathcal{S}=\{X\}$ and the Whitney stratification of a variety $X$ whose singular locus is a finite set of point is $\mathcal{S}=\{X_{\mathrm{reg}},\{s_1\},\ldots,\{s_r\}\}$, where $X_{\mathrm{reg}}$ is the set of smooth points of $X$ and $s_1,\ldots,s_r$ are the singular points of $X$. By a theorem of Whitney, any algebraic variety has a Whintey stratification \cite{whitney1992local,whitney1992tangents}.

Before we state the main theorem on Whitney stratifications that we use in this article, we define \textit{Milnor fibers} \cite[Chapter 10]{maxim2019intersection}. Let $X$ be a smooth variety and $V$ a divisor on $X$. Choose any Whitney stratum $S \in \mathcal S$ and any point $x \in S$. In a
sufficiently small ball $B_{\epsilon,x}$ centered at $x$, the hypersurface $V$ is equal to the zero locus of a holomorphic function $f$. The Milnor fiber of $V$ at $x \in  S$ is given by
\begin{align}   
F_x:=B_{\epsilon,x} \cap \{f = t\},
\end{align}
for small $|t|$ greater than 0. The Euler characteristic of $F_x$ is independent of the choice 
of the local equation $f$ at $x$, and it is constant along the given stratum containing $x$.

\begin{theorem}[{\hspace{1sp}\cite{parusinski1995formula},\cite[Theorem 10.4.4]{maxim2019intersection}}] \label{thm: Whitney} Let $X$ be a smooth complex projective variety,
and let $V$ be a very ample divisor in $X$. Let $\sqcup_{s\in \mathcal{S}} S$ be a Whitney stratification
of $V$. Let $W$ be another divisor on $X$ that is linearly equivalent to $V$. Suppose $W$ is
smooth and $W$ intersects $V$ transversally in the stratified sense (with respect to the above
Whitney stratification). Then we have
\begin{align}
    \chi(W)-\chi(V)=\sum_{S\in \mathcal S} \mu_S \chi(S\setminus W), 
\end{align}
where $\mu_S$ is the Euler characteristic of the reduced cohomology of the Milnor fiber at any
point $x \in S$.
\end{theorem}

The Euler characteristic of the \textit{reduced cohomology} is the normal Euler characteristic minus one.


\section{Computation of Euclidean Distance Degrees}\label{s: EDD} 

Let $X\in \PP^3$ be a point and $L\in \Gr(1,\PP^3)$ a line. Let $\mathcal
C$ be a generic arrangement of $m$ cameras. For the sake of notation, we write $\mathcal M_m^L$ for $\mathcal{M}_\mathcal{C}^L$ and $\mathcal{L}_m^X$ for $\mathcal{L}_\mathcal{C}^X$. We assume from now on that $m\ge 3$ for the anchored line multiview variety $\mathcal{L}_m^X$. We recall notation from \cite{EDDegree_point}. Write each $\PP^2$ as $\CC^2 \cup\PP_\infty^1$, where $\CC^2$ is the chosen affine chart and $\PP_\infty^1$ is the line at
infinity. Denote the hypersurface $\PP^2\times \cdots\times \PP_\infty^1\times \cdots\times \PP^2$ in $(\PP^2)^m$ by $H_{\infty,i}$, where $\PP_\infty^1$ is the $i$-th factor. Let $H_\infty=\cup_{i=1}^m H_{\infty,i}$. Denote by $H_Q$ the closure of the
hypersurface $\sum_{i=1}^{2m}(z_i - \beta_i)^2+\beta_0=0$ in $(\PP^2)^m$. In the remainder of this proof, we will
use the following notation:
\begin{align}\begin{aligned}
    D_Q^L &:= \mathcal{M}_m^L \cap  H_Q, D_{\infty,i}^L := \mathcal{M}_m^L \cap H_{\infty,i},\\
    D_\infty^L &:= \mathcal{M}_m^L\cap H_{\infty},    
\end{aligned}
\end{align}
for the anchored point multiview variety, and
\begin{align}\begin{aligned}
    D_Q^X &:= \mathcal{L}_m^X \cap  H_Q, D_{\infty,i}^X := \mathcal{L}_m^X \cap H_{\infty,i},\\
    D_\infty^X &:= \mathcal{L}_m^X\cap H_{\infty},
\end{aligned}
\end{align}
for the anchored line multiview variety. Write $M_m^L$ and $L_m^X$ for the corresponding affine varieties. Notice that $H_\infty$ is the complement of the affine chart $\CC^{2m}$ in $(\PP^2)^m$, thus $D_\infty^L$ is the
complement of $M_m^L$ in $\mathcal{M}_m^L$ and $D_\infty^X$ is the
complement of $L_m^X$ in $\mathcal{L}_m^X$ and.
As derived in \cite{EDDegree_point}, we have,
\begin{align}
    \chi(M_m^L\cap U_\beta)=&\label{eq:eulML1}\\
    \chi(\mathcal M_m^L)-\chi(D_\infty^L)&+\chi(D_Q^L\cap D_\infty^L)-\chi(D_Q^L),\label{eq:eulML2}\\
    \chi(L_m^X\cap U_\beta)=&\label{eq:eulLX1}\\
    \chi(\mathcal{L}_m^X)-\chi(D_\infty^X)&+\chi(D_Q^X\cap D_\infty^X)-\chi(D_Q^X).\label{eq:eulLX2}
\end{align}
The structure of the proof of \Cref{thm: EDDthm} is to calculate the four terms of \Cref{eq:eulML2,eq:eulLX2}. 

\begin{lemma}\label{lem: anchored-chiY_n} For a fixed $X$ and $L$, let $\mathcal{C}$ be a generic arrangement of $m$ cameras.

\begin{enumerate}
    \item $\chi(\mathcal{M}_\mathcal{C}^L)=2$;
    \item $\chi(\mathcal{L}_\mathcal{C}^X)=3+m$.
\end{enumerate}
\end{lemma}

\begin{proof} $ $

$\textit{1}.$ $\mathcal{M}_\mathcal{C}^L$ is isomorphic to $\PP^1$ and we are done by \Cref{le: Eul3}.

$\textit{2}.$ Recall that we assume $m\ge 3$. By genericity, $c_i$ are not collinear. Therefore the back-projected planes of an element $\ell\in \mathcal{L}_\mathcal{C}^X$ meet in exactly a line. Consider the partition 
\begin{align}\mathcal{L}_\mathcal{C}^X= U\cup \bigcup_{i=1}^m U_i,
\end{align}
where $U$ is the set of $\ell\in \mathcal{L}_\mathcal{C}^X$ whose back-projected planes meet in a line away from any center, and $U_i$ is the set of $\ell\in \mathcal{L}_\mathcal{C}^X$ whose back-projected planes meet in $c_i\vee X$. By \Cref{le: Eul1}, $\chi(\mathcal{L}_\mathcal{C}^X)=\chi(U)+\sum \chi(U_i)$. Since $\Upsilon_\mathcal{C}$ is injective on the subset of $\Lambda(X)\setminus\cup ( c_i\vee X)$, we see via the isomorphism $U\cong \PP^2$ that $U$ is isomorphic to $\PP^2$ minus $m$ points. By \Cref{le: Eul2}, $\chi(U)=3-m$. If instead $\ell\in \mathcal{L}_\mathcal{C}^X$ meet exactly in $c_i\vee X$, then for $j\neq i$ we have $\ell_j=C_j\cdot (c_i\vee X)$, and $\ell_i$ is any line in $\Lambda(C_iX)$. However, $\Lambda(C_iX)\cong \PP^1$, implying $\chi(U_i)=2$. In total, $\chi(\mathcal{L}_\mathcal{C}^X)=3-m+2m=3+m$.
\end{proof}

In the next step, we compute the second terms of the right-hand sides of \Cref{eq:eulML2,eq:eulLX2}. 

\begin{lemma}\label{lem: anchored-chiDinfty} For a fixed $X$ and $L$, let $\mathcal C$ be a generic arrangement of cameras of $m$ cameras. Then

\begin{enumerate}
    \item $\chi(D_\infty^L)=m$;
    \item $\chi(D_\infty^X)=2m-\binom{m}{2}$.
\end{enumerate}
\end{lemma}

\begin{proof} $ $

$\textit{1}.$ Each $D_{\infty, i}^L$ is a generic point of $\mathcal{M}_m^L$. By additivity of the Euler characteristic, we have that 
\begin{align}
    \chi(D_{\infty}^L)=\chi(\cup_{i=1}^m D_{\infty, i}^L) =\sum _{i=1}^m\chi( D_{\infty, i}^L)=  m.
\end{align}

$\textit{2}.$ Each $D_{\infty, i}^X$ is a curve inside $\mathcal{L}_m^X$. This curve corresponds precisely to fixing the $i$-th back-projected plane $H_i$ to be generic through $c_i$ and $X$. Such a plane contains no other center, and a line in this plane uniquely determines all other back-projected planes. Therefore $D_{\infty, i}^L$ is isomorphic to the set of lines in $H_i$ through $X$, which in turn is isomorphic to $\mathbb{P}^1$. We get $\chi(D_{\infty,i}^X) = 2$. Moreover, $\chi(D_{\infty,i}^X)$ only have pairwise intersections, and two generic back-projected planes $H_i,H_j$ through $c_i,X$ and $c_j,X$, respectively, meet in exactly a generic line through $X$. Therefore $D_{\infty,i}^X \cap D_{\infty,j}^X $ consists of a single element. We then get
\begin{align}
    \chi(D_\infty^X)&=\chi(\cup_{i=1}^m D_{\infty,i}^X) \\
    &=\sum_{i=1}^m \chi( D_{\infty,i}^X) - \sum_{i<j} \chi\left( D_{\infty,i}^X \cap D_{\infty,j}^X \right) \\
    &= \sum_{i=1}^m 2 - \sum_{i<j} 1 =  2m - \binom{m}{2},
\end{align}
by additivity.
\end{proof}
We recall that $H_Q$ is the closure of the affine hypersurface
\begin{align}\label{eq: quad}
 \sum_{1\le i\le 2m}
(z_i-\beta_i)^2+ \beta_0 = 0,
\end{align}
in $(\PP^2)^m$. We introduce homogeneous coordinates $x_i, y_{2i-1}, y_{2i}$ with $z_{2i-1} =y_{2i-1}/x_i$ and $z_{2i} =y_{2i}/x_i$ for $1\le i\le m$. Write $\textbf{x}=x_1\cdots x_m$. Then the homogenization of \Cref{eq: quad}, and hence the equation of
$H_Q$, is
\begin{align}\label{eq: homog1}
\begin{aligned}
   \sum_{i=1}^m \Big(&(y_{2i-1}- \beta_{2i-1}x_i)^2
 + (y_{2i}- \beta_{2i}x_i)^2\Big)\frac{\textbf{x}^2}{x_i^2}+\\
 +&\beta_0 \textbf{x}^2=0. 
\end{aligned}
\end{align}

\begin{lemma}
\label{lem: achored-chiDinftycapDQ}
$ $
\begin{enumerate}
    \item $\chi(D_Q^L\cap D_\infty^L)=0$;
    \item $\chi(D_Q^X\cap D_\infty^X)=\binom{m}{2}$.
\end{enumerate}

\end{lemma}

\begin{proof} We homogenize the equation defining $H_Q$ as in \Cref{eq: homog1}, and assume without loss of generality that $H_{\infty,i}$ is defined by the equation $x_i=0$. We have by inspection,
\begin{align}
    H_Q\cap H_{\infty,i}=&\{y_{2i-1}+\sqrt{-1}y_{2i}=x_i=0\}\cup\\
    \cup &\{y_{2i-1}-\sqrt{-1}y_{2i}=x_i=0\}\cup \\
    \cup&\bigcup_{j\neq i}\{x_i=x_j=0\}.
\end{align}	
Let,
\begin{align*}
    	K_i^{L,+}:= &\mathcal{M}_m^L \cap \{y_{2i-1}+\sqrt{-1}y_{2i}=x_i=0\},\\
		K_i^{L,-}:= &\mathcal{M}_m^L \cap \{y_{2i-1}-\sqrt{-1}y_{2i}=x_i=0\}, \\
		A_{i,j}^L:= &\mathcal{M}_m^L \cap \{x_i=x_j=0\}, j\neq i.
\end{align*}
Write $K_i^{X,+},K_i^{X,-}$ and $A_{i,j}^X$ analogously for intersecting with $\mathcal{L}_m^X$ instead of $\mathcal{M}_m^L$. With this notation,
\begin{align}
    \mathcal{M}_m^L\cap H_Q\cap H_{\infty,i}=K_i^{L,+} \cup K_i^{L,-} \cup \bigcup_{i\neq j} A_{i,j}^L,\\
      \mathcal{L}_m^X\cap H_Q\cap H_{\infty,i}=K_i^{X,+} \cup K_i^{X,-} \cup \bigcup_{i\neq j} A_{i,j}^X.
\end{align}
As we shall see below, $K_i^{L,\pm},K_i^{X,\pm}=\emptyset$. Therefore, by inclusion/exclusion, 
\begin{align}
    \chi(D_Q^L\cap D_\infty^L)&=\chi(\bigcup_{i\neq j} A_{i,j}^L),\label{eq: AijL}\\
    \chi(D_Q^X\cap D_\infty^X)&=\chi(\bigcup_{i\neq j} A_{i,j}^X).\label{eq: AijX}
\end{align}

$\textit{1}.$ By construction, the $i$-th factor of any element of $K_i^{L,\pm}\subseteq (\PP^2)^m$ is fixed equal to $[0:\mp\sqrt{-1}:1]$. However, by genericity, this point does not lie in $C_i\cdot L$, implying that $K_i^{L,\pm}=\emptyset$. Regarding $A_{i,j}^L$, setting $x_i=0,x_j=0$ corresponds to fixing two generic image lines in the corresponding image planes. The back-projected planes of those image lines meet in a generic line, and such a line does not meet $L$. This implies $A_{i,j}^L=\emptyset$, and we are done by \Cref{eq: AijL}.

$\textit{2}.$ By construction, the $i$-th factor of any element of $K_i^{X,\pm}\subseteq (\PP^2)^m$ is fixed equal to $[0:\mp\sqrt{-1}:1]$. However, by genericity, the line this vector defines does not contain $C_iX$, implying that $K_i^{X,\pm}=\emptyset$. Regarding $A_{i,j}^X$, setting $x_i=0,x_j=0$ corresponds to intersecting $\Lambda(C_iX),\Lambda(C_jX)$ with generic hyperplanes. They intersect in the single elements $\ell_i,\ell_j$. The  back-projected planes of $\ell_i,\ell_j$ meet in a generic line through $X$. Therefore $A_{i,j}^X$ is a generic point of $\mathcal{L}_m^X$. All $A_{i,j}^X$ are disjoint, and we are done by \Cref{eq: AijX}.
\end{proof}


The hypersurface $H_Q$ in the hypersurface $H_Q$ is defined by \Cref{eq: homog1}. It follows by \Cref{ss: Chow} that we have the following linear equivalence of divisors in $(\PP^2)^m$: 
\begin{align}
    H_Q\equiv 2H_{\infty,1}+\cdots +2H_{\infty,m}.
\end{align}
Then as divisors of the anchored multiview varieties,
\begin{align}
    D_Q^L&\equiv 2 \mathcal{M}_m^L\cap H_{\infty,1}+\cdots+2 \mathcal{M}_m^L\cap H_{\infty,m},\\
    D_Q^X&\equiv 2 \mathcal{L}_m^X\cap H_{\infty,1}+\cdots+2 \mathcal{L}_m^X\cap H_{\infty,m}.
\end{align}
Consider the well-defined projections $\pi_L:\mathcal{M}_m^L\to L$, and $\pi_X: \mathcal{L}_m^X\to \Lambda(X)$, sending image points and image lines to the intersection of their back-projected lines or planes. In the Chow ring of $\PP^3$, every element of $L$ is equivalent. We denote by $D_H^L$ the preimage of a generic hyperplane in $L$, i.e. a generic point of $L$. In the Chow ring of $\Gr(1,\PP^3)$, every element of $\Lambda(X)$ is equivalent. In particular, we have $\pi_X^*[H]=\pi_X^*[H']$, where $H, H'$ are hyperplanes of lines in $\Lambda(X)$, where a hyperplane of lines is the set of lines in $\Lambda(X)$ contained in some hyperplane of $\PP^3$ through $X$. Let $H$ be a generic plane of lines in $\Lambda(X)$ and $D_H^X=\pi_X^{*}(H)$. Let $H'$ be a generic among the planes of lines that contain $c_i\vee X$ for some $i$. In the Chow ring of $\mathcal L_m^X$,  $\pi_X^{*}(H')$ is the union of $\mathcal{L}_m^X\cap H_{\infty,i}$ and the variety $E_i^X$ of elements $\ell$ whose back-projected planes meet exactly in the line $c_i\vee X$. In other words, $E_i^X=\pi_X^{*}(c_i\vee X)$. We get
\begin{align}\begin{aligned}
\mathcal{M}_m^L\cap H_{\infty,i}&\equiv D_H^L,\\
\mathcal{L}_m^X\cap H_{\infty,i}&\equiv D_H^X-E_i^X.
\end{aligned}
\end{align}
It follows that,
\begin{align}\label{eq: Dqx}
\begin{aligned}
D_Q^L&\equiv 2mD_H^L,\\
    D_Q^X&\equiv 2mD_H^X-2E_i^X-\cdots -2E_i^X.
\end{aligned}
\end{align}
Note that $E_i^X\not\equiv E_j^X$ for $i\neq j$. 

\begin{lemma}\label{le: ChowID}$ $

\begin{enumerate}
    \item $[D_H^L]^2=0$;
    \item $[E_i^X]^3=0$;
    \item $[D_H^X]^3=0$;
    \item $[D_H^X][E_i^X]=0$ for $i\neq j$;
    \item $[E_i^X][E_j^X]=0$ for $i\neq j$.
\end{enumerate}
\end{lemma}

\begin{proof} $ $

$\textit{1,2,3.}$ This is a consequence of the fact that $D_H^L$ is a proper subvariety of an irreducible variety of dimension 1. Similarly, both $E_i^X$ and $D_H^X$ are proper subvarieties of an irreducible variety of dimension 2.

$\textit{4}.$ For a generic plane of lines $H$ through $X$, $D_H^X\cap E_i^X$ is empty. This suffices by \Cref{thm: prod}. 

$\textit{5}.$ We use the fact that $E_j^X\equiv D_H^X-\mathcal{L}_m^X\cap H_{\infty,j}$. By~$\textit{4}.$, intersecting the right-hand side with $E_i^X$ yields the following $E_i^X\cap (\mathcal{L}_m^X\cap H_{\infty,j})$. Now the $j$-th factor of $\mathcal L_m^X\cap H_{\infty,j}$ consists of a fixed generic line through $C_iX$. Its back-projected plane does not contain $c_i\vee X$. It follows that the intersection must be empty. This suffices by \Cref{thm: prod}. 
\end{proof}

\begin{proposition}\label{prop: Chow} In the chow ring of $(\PP^2)^m$, we have the following identities.
\begin{enumerate}
    \item  \label{eq:chernM}
        $c(\mathcal{M}_m^L)=1+2[D_H^L]$;
         
    \item \label{eq:chernL} $
        c(\mathcal{L}_m^X)=(1+[D_H^X])^3-\sum_{i=1}^m \Big([E_i^X]+[E_i^X]^2\Big).$
\end{enumerate}
\end{proposition}

\begin{proof} $ $

$\textit{1}.$ Follows from the fact that $\mathcal{M}_m^L$ is isomorphic to $\PP^1$, which in turn has Chern class $1+2[x]$ by \Cref{le: ChernP}, where $[x]$ represents a point of $\PP^1$.

$\textit{2}.$ Recalling that we in the proof of \Cref{prop: smooth} viewed $\mathcal{L}_m^X$ as a blow-up, the Chern class formula from \Cref{ss: Chern} gives us 
\begin{align}
    c(\mathcal{L}_m^X)=&(1+[D_H^X])^3+\\
    &+\sum_{i=1}^m \Big((1+[E_i^X])(1-[E_i^X])^2-1\Big).
\end{align}
After simplification and the fact that $[E_i^X]^3=0$, we get the statement.
\end{proof}

As a sanity check, we note that \Cref{prop: Chow} gives us the correct Euler characteristics via the Chern-Gauss-Bonnet theorem. The top Chern class of $\mathcal{M}_m^L$ is $2[D_H^L]$, where $[D_H^L]$ is the class of a single point. The top Chern class of $\mathcal{L}_m^X$ is $3[D_H^X]^2+\sum -[E_i^X]^2$. Now $[D_H^X]^2$ corresponds to the preimage of the intersection of two generic planes through $X$; it corresponds to a single point. Next, due to the fact that $E_i^X\equiv D_H^X-\mathcal{L}_m^X\cap H_{\infty,i}$, we have that $[E_i^X]^2\equiv [D_H^X]^2-2[D_H^X\cap \mathcal{L}_m^X\cap H_{\infty,i}]+[\mathcal{L}_m^X\cap H_{\infty,i}]^2$. However, $[D_H^X\cap \mathcal{L}_m^X\cap H_{\infty,i}]$ is a single point and the intersection $(\mathcal{L}_m^X\cap H_{\infty,i})^2$ is empty. Therefore $[E_i^X]^2$ is equal to minus a single point. The Chern-Gauss-Bonnet theorem then states that
\begin{align}\begin{aligned}
    \chi(\mathcal{M}_m^L)&=2,\\
     \chi(\mathcal{L}_m^X)&=3+m.    
\end{aligned}
\end{align}

We aim to use \Cref{thm: Whitney} to determine the Euler characteristic of $D_Q^X$ and $D_Q^L$. We start by considering generic divisors in their linear systems.

\begin{lemma} $ $

\begin{enumerate}
     \item A generic divisor $D'^{L}$ in the linear system $\Gamma (\mathcal{M}_m^L,\mathcal O(D_Q^L ))$ is smooth;
    \item A generic divisor $D'^X$ in the linear system $\Gamma (\mathcal{L}_m^X,\mathcal O(D_Q^X ))$ is smooth.
    \end{enumerate}
\end{lemma}
\begin{proof} We recall that
\begin{align}
    H_Q\equiv 2H_{1,\infty}+\cdots +2H_{m,\infty}.
\end{align}
Any variety that is the union of hypersurfaces from $i=1,\ldots,m$ of double lines in factor $i$ is linearly equivalent to $H_Q$. It is clear that intersecting all such unions gives the empty set, and therefore the base locus of the divisor $H_Q$ is empty. By \Cref{le: basepoint-free}, the restriction of $H_Q$ to $\mathcal M_m^L$ and $\mathcal L_m^X$ gives a basepoint-free linear system. We are done by Bertini's theorem.
\end{proof}

\begin{proposition}\label{prop: eulgendiv} $ $

\begin{enumerate}
    \item If $D'^{L}$ is a generic divisor in the linear system $\Gamma (\mathcal{M}_m^L,\mathcal O(D_Q^L ))$, then $\chi(D'^L)=2m$;
    \item If $D'^X$ is a generic divisor in the linear system $\Gamma (\mathcal{L}_m^X,\mathcal O(D_Q^X ))$, then $\chi(D'^X)=-4m^2+8m$.
\end{enumerate}
\end{proposition}

\begin{proof} $ $

$\textit{1}.$ For the purpose of applying the Chern-Gauss-Bonnet theorem, we want to find the top Chern class of $D'^L$, which is $c_0(D'^L)=1$. Using $D'^X\equiv D_H^X$, it follows that $\chi(D'^L)$ equals
\begin{align}
    \int_{D'^L} 1|_{D'^L} =\int_{\mathcal{M}_m^L}1 [D'^X]= \int_{\mathcal{M}_m^L} 2m[D_H^X] ,\label{eq: finalCGBpoint}
\end{align}
which equals $2m$ since $[D_H^X]$ is the class of one simple point in $\mathcal M_m^L$.

$\textit{2}.$ By the adjunction formula and considering the fact that $D'^X\equiv D_Q^X$, we have 
\begin{align}\label{eq: D'rest}
\begin{aligned}
     c(D'^X) &= \frac{c(\mathcal L_m^X|_{D'^X})}{(1+[D'^X])|_{D'^X}}\\
     &=\Big(c(\mathcal L_m^X)(1+[D_Q^X])^{-1}\Big)\Big|_{D'^X}.
\end{aligned}
\end{align}
Throughout this proof, we use \Cref{le: ChowID}. The identity $(1+u)^{-1}=1-u+u^2-\cdots$ and \Cref{eq: Dqx} imply  
\begin{align}\label{eq: 1+inverse}
\begin{aligned}
     (1+[D_Q^X])^{-1}=&1-2m[D_H^X]+2\sum [E_i^X]+\\
     &+4m^2[D_H^X]^2+4\sum [E_i^X]^2.    
\end{aligned}
\end{align}
For the purpose of applying the Chern-Gauss-Bonnet theorem, we want to find the top Chern class of $D'^X$, which is $c_1(D'^X)$. However, this is the first Chern class of $c(\mathcal L_m^X)(1+[D_Q^X])^{-1}$ restricted to $D'^X$, by \Cref{eq: D'rest}. Now using \Cref{eq: 1+inverse} and \Cref{prop: Chow}, the first Chern class of $c(\mathcal L_m^X)(1+[D_Q^X])^{-1}$ can be written
\begin{align}
(-2m+3)[D_H^X] +\sum [E_i^X].
\end{align}
It follows that $\chi(D'^X)$ equals
\begin{align}
\begin{aligned}
    &\int_{D'^X} \Big((-2m+3)[D_H^X] +\sum [E_i^X]\Big)\Big|_{D'^X}=\\
    = &\int_{\mathcal{L}_m^X} \Big((-2m+3)[D_H^X] +\sum [E_i^X]\Big)[D'^X]\\
    = &\int_{\mathcal{L}_m^X}(-4m^2+6m)[D_H^X]^2 -2\sum [E_i^X]^2,\label{eq: finalCGB}
\end{aligned}
\end{align}
where we in the last equality used \Cref{eq: Dqx}. Recall then that $\deg [D_H^X]^2=1$ and $\deg [E_i^X]^2=-1$. So \Cref{eq: finalCGB} adds to $-4m^2+6m+2m$. 
\end{proof}

If we consider $y_{2i-1}, y_{2i}$ and $x_i$ as sections of line bundles $\mathcal{M}_m^L$ and $\mathcal L_m^X$, then $D_Q^L = \mathcal{M}_m^L \cap H_Q$, respectively $D_Q^X = \mathcal{L}_m^X \cap H_Q$, is a
general divisor in the linear system given by the subspace $\Gamma'^L$ of $\Gamma(\mathcal{M}_m^L, \mathcal O(2mD_H^L))$, respectively $\Gamma'^X$ of $\Gamma(\mathcal{L}_m^X, \mathcal O(2mD_H^X-2E_i^X-\cdots-2E_i^X))$, generated by the sections
\begin{align}\label{eq: gen-sec}
\begin{aligned}
   1.\; &(y_1^2 + y_2^2) x_2^2\cdots x_m^2 +\cdots + \\
   &+(y_{2m-1}^2 + y_{2m}^2) x_1^2\cdots x_{m-1}^2,\\
   2.\; &x_1^2\cdots x_m^2,\\
   3.\; & \frac{y_{2i-1}}{x_i}x_1^2\cdots x_m^2 \textnormal{ for }i=1,\ldots,m,\\
   4.\; & \frac{y_{2i}}{x_i}x_1^2\cdots x_m^2 \textnormal{ for }i=1,\ldots,m.
\end{aligned}
\end{align}
To be precise, $D_Q^L$ and $D_Q^X$ are defined through global sections that determined are by $H_Q$, and $H_Q$ is a linear combination of the generators of \Cref{eq: gen-sec} with generic coefficients as in \Cref{eq: homog1}.

\begin{proposition} $ $

\begin{enumerate}
    \item $D_Q^L$ is smooth;
    \item The singular locus of $D_Q^X$ is the set of $\binom{m}{2}$ points
\begin{align}\label{eq: singloc}
    \bigcup_{i\neq j} D_{\infty,i}^X\cap D_{\infty,j}^X.
\end{align}
\end{enumerate}

\end{proposition}

\begin{proof} $ $ 

$\textit{1}.$ The base locus of $\Gamma'^L$ is $\cup D_{\infty,i}^L\cap D_{\infty,j}^L$. Indeed, this is precisely the zero locus of the polynomials of \Cref{eq: gen-sec}. However, each $D_{\infty,i}^L\cap D_{\infty,j}^L$ is empty. By Bertini's theorem, $D_Q^L$ is smooth away from this empty set. 

$\textit{2}.$ Similarly, the base locus of $\Gamma'^X$ is $\cup D_{\infty,i}^X\cap D_{\infty,j}^X$, and each $D_{\infty,i}^X\cap D_{\infty,j}^X$ is a point. On the other hand, $D_Q^X$
has multiplicity at least 2 along $\cup D_{\infty,i}^X\cap D_{\infty,j}^X$. We can see this by looking at the Jacobian condition. The vanishing ideal of $\mathcal{L}_m^X$ together with the additional equation of $H_Q$ defines $D_Q^X$, a variety of dimension $1$. At $S_{i,j}$, the gradient of the generators of $\mathcal{L}_m^X$ give the correct corank $2$ since it is smooth, but the additional equation has gradient zero so that the corank is not equal to 1. 
\end{proof}

\begin{proposition} $ $

\begin{enumerate}
    \item A Whitney stratification of $D_Q^L$ is the single stratum $S_{\textnormal{reg}}=D_Q^L$,
    \item A Whitney stratification of $D_Q^X$ consists of the stratum of smooth points $S_{\textnormal{reg}}$ and $S_{i,j}=D_{\infty,i}^X\cap D_{\infty,j}^X $.
\end{enumerate}\label{prop: Whitney}
\end{proposition}

\begin{proof} This is stated in \Cref{ss: Whitney}.
\end{proof}

\begin{proposition}\label{prop: redMilnor} The Euler characteristics of the reduced cohomology of the Milnor fibers of the points in \Cref{eq: singloc} are $-1$. 
\end{proposition}

\begin{proof} Near
\begin{align}
    S_{i,j}=D_{\infty,i}^X\cap D_{\infty,j}^X,
\end{align}
the functions $x=\frac{x_i}{y_{2i}},y=\frac{x_j}{y_{2j}}$ form a coordinate frame of $\mathcal{L}_{m}^X$, meaning the values of $x,y$ determine a unique point of $\mathcal{L}_{m}^X$. This translates to $D_Q^X$ being determined by the equation
\begin{align}
u_1x^2+u_2y^2=u_3x^2y^2,
\end{align}
for holomorphic locally non-vanishing functions $u_1,u_2,u_3$ that we can read off from the homogenization of $H_Q$ in \Cref{eq: homog1}.

Next look at $G_t=\{x^2+y^2-x^2y^2=t\}\cap B_{\epsilon}$ and the map
\begin{align}\begin{aligned}
    \psi:G_t&\to \{x+y-xy=t\}\cap B_{\epsilon^2},\\
    (x,y)&\mapsto (x^2,y^2).    
\end{aligned}
\end{align}
Denote by $G_t'$ the set $ \{x+y-xy=t\}\cap B_{\epsilon^2}$. Consider the disjoint union
\begin{align}\begin{aligned}
    G_t'=&(G_t'\cap \{x,y\neq 0\})\cup (G_t'\cap \{x=0\})\\
    &\cup (G_t'\cap \{y= 0\}).
\end{aligned}
\end{align}
Note that $G_t'$ is smooth at every point for small $\epsilon$; the gradient is $(1-x,1-y)$. Observe that $G_t'\cap \{x=0\}$ and $G_t'\cap \{y=0\}$ are by construction single points. We have that $\psi$ is 4-to-1 on the first set and 2-to-1 on the second and third of the disjoint union. This gives us 
\begin{align}\begin{aligned}
    \chi(G_t)&=4\chi((G_t'\cap \{x,y\neq 0\})+\\
    &+2 \chi(G_t'\cap \{x=0\})+2\chi (G_t'\cap \{y= 0\})\\
    &=4(1-2)+2+2=0.   
\end{aligned}
\end{align}
We conclude that the Euler characteristic of the reduced cohomology of the Milnor fiber is $-1$.
\end{proof}

\begin{proof}[Proof of \Cref{thm: EDDthm}] $ $

$\textit{1}.$ We use \Cref{thm: Whitney} and \Cref{prop: Whitney} to conclude that 
\begin{align}
\chi(D_Q^L)=\chi(D'^L).
\end{align}
We get by \Cref{eq:eulML1,eq:eulML2}, and \Cref{lem: anchored-chiY_n,lem: anchored-chiDinfty,lem: achored-chiDinftycapDQ} and \Cref{prop: eulgendiv} that
\begin{align}
    \chi(M_m^L\cap U_\beta)= 2-m+0-2m,
\end{align}
which sums to $2-3m$. Since $\dim \mathcal{M}_m^L=1$, \Cref{thm: EDDthm} says that $\mathrm{EDD}(\mathcal{M}_m^L)=3m-2$.

$\textit{2}.$ Similarly, by \Cref{thm: Whitney} and \Cref{prop: Whitney}, we get  
\begin{align}\begin{aligned}
\chi(D'^X)-\chi(D_Q^X)=&\mu_0 \chi(S_{\textnormal{reg}}\setminus D'^X)+\\
&+\sum_{i\neq j}\mu_{i,j} \chi(S_{i,j}\setminus D'^X), \end{aligned}
\end{align}
where $\mu_0$ and $\mu_{i,j}$ are defined as in \Cref{thm: Whitney}. It is not hard to check that $\mu_0$. Observe that $D'^X$ does not meet any singular points, for instance since the linear system $\Gamma(\mathcal{L}_m^X,\mathcal O(D_Q^X))$ is basepoint-free. Therefore $\chi(S_{i,j}\setminus D'^X)=\chi(S_{i,j})=1$. We get by \Cref{eq:eulML1,eq:eulML2}, and \Cref{lem: anchored-chiY_n,lem: anchored-chiDinfty,lem: achored-chiDinftycapDQ} and \Cref{prop: eulgendiv,prop: redMilnor} that 
\begin{align}
    \chi(L_m^X\cap U_\beta)&= (3+m)-(2m-\binom{m}{2})+\\
    &+\binom{m}{2}-(-4m^2+8m+\binom{m}{2}),
\end{align}
which sums to $\frac{9}{2}m^2-\frac{19}{2}m+3$. Since $\dim \mathcal{L}_m^X=2$, \Cref{thm: EDDthm} says that $\mathrm{EDD}(\mathcal{L}_m^X)=\frac{9}{2}m^2-\frac{19}{2}m+3 $.
\end{proof}

The following is now a direct consequence:

\begin{customcorollary}{\ref{cor: EDD-p}}  Let $\widetilde{\mathcal{C}}$ and $\widehat{\mathcal{C}}$ be generic arrangements of cardinality $m$.
\begin{enumerate}
    \item $\mathrm{EDD}( \mathcal{M}_{\widetilde{\mathcal{C}}}^{1, 1})=3m-2 $.
    \item If $m\ge 3$, then $\mathrm{EDD}( \mathcal{M}_{\widehat{\mathcal{C}}}^{2,1})=\frac{9}{2}m^2-\frac{19}{2}m+3.$
\end{enumerate}
\end{customcorollary}

\begin{proof} Follows from \Cref{thm: EDDthm} and \Cref{thm: EDbij}.
\end{proof}


\section{Pseudocodes}\label{s: Pseudo} Finally, in the last section we provide the pseudocode that lay the foundation for our numerical results. For each of the plots presented in the main document, we iterate 1000 (or 100) times on each of the 5 different ways of reconstruction and then we plot the relative error and speed. Note that each time we generate new random camera arrangements, a line in $\RR^3$ and $p$ points on this line.


In the pseudocodes below, we present one iteration of each method. The input is a randomly generated camera arrangement $\mathcal{C}$ of $3\times 4$ matrices, a projective line $L$ spanned by two vectors of $\RR^4$, and $p$ points $X_i\in \RR^3$ such that $[X_i;\; 1]$ lie on $L$. We use the notation that for a column vector $X\in \RR^n$, $[X; \; 1]\in \RR^{n+1}$ is the vector we get by adding a $1$ as the last coordinate. Let $\iota$ be the function that scales a vector such that its last coordinate is 1, and then removes that coordinate. When we write $L':[L';\;1]\in \Gr(1,\PP^3)$, we mean that $L'$ is a line spanned by two column vectors $l_0,l_1$ such that the $2\times 2$ lower minor of $\begin{bmatrix} l_0 & l_1 \end{bmatrix}$ is non-zero. This corresponds to choosing an affine patch of the Grassmannian of lines in $\PP^3$. 


In \Cref{alg:L1.0,alg:L1.1,alg:L1.2,alg:L1.3,alg:L1.4} we use the standard approach for simplicity, but we provide in \Cref{alg:L1.1 nonstd} the non-standard approach for (L1).1 to emphasize the distinction.

\begin{algorithm}
\SetKwInOut{Input}{Input}\SetKwInOut{Output}{Output}
\SetKwInOut{Return}{Return}
\caption{Method (L1).0.}
\label{alg:L1.0}
\Input{$\mathcal{C}=(C_1,\ldots,C_m)$, $X_1,\ldots,X_p$}
\Output{The log of the average relative error} 
  \For{$j$ \textnormal{from $1$ to $m$}}{
    \For{$i$ \textnormal{from $1$ to $p$}}{
        $q_{i,j} \gets \iota(C_j[X_i;\;1])  +\sigma(\epsilon$)\;}}
$Y_i \gets \underset{{X\in \RR^3}}{\mathrm{argmin}} \sum^{m}_{j=1} (q_{i,j}-\iota(C_j[X;\;1]))^2$\;
$e\gets \log_{10}\left(\frac{1}{p\epsilon}\sum_{i=1}^p \|Y_i-X_i\|\right)  $\;
\Return{$e$}

 \end{algorithm}

 \begin{algorithm}
\SetKwInOut{Input}{Input}\SetKwInOut{Output}{Output}
\SetKwInOut{Return}{Return}
\caption{Method (L1).1 std.}
\Input{$\mathcal{C}=(C_1,\ldots,C_m)$, $L$, $X_1,\ldots,X_p$}
\Output{The log of the average relative error} 
  \For{$j$ \textnormal{from $1$ to $m$}}{
    \For{$i$ \textnormal{from $1$ to $p$}}{
        $q_{i,j} \gets \iota(C_j[X_i;\;1]) +\sigma(\epsilon$)\;}
        $u_j \gets \iota(C_j \cdot L) +\sigma(\epsilon)$\;}
    $L_0\gets \mathrm{nullspace} \begin{bmatrix}
    C^T_1[u_1;\;1]&C^T_2[u_2;\;1]
    \end{bmatrix}^T$\;
    \For{$i$ from 1 to $p$}{$Y_i \gets \underset{{X\in \RR^3: [X;\;1]\in L_0}}{\mathrm{argmin}} \sum^{m}_{j=1} (q_{i,j}-\iota(C_j[X;\;1]))^2$\; }
    $e\gets \log_{10}\left(\frac{1}{p\epsilon}\sum_{i=1}^p \|Y_i-X_i\|\right)  $\;

\Return{$e$}

\end{algorithm}

\begin{algorithm}
\SetKwInOut{Input}{Input}\SetKwInOut{Output}{Output}
\SetKwInOut{Return}{Return}
\caption{Method (L1).2 std.}
\label{alg:L1.2}
\Input{$\mathcal{C}=(C_1,\ldots,C_m)$, $X_1,\ldots,X_p$}
\Output{The log of the average relative error}
    \For{$j$ from 1 to $m$}{
    \For{$i$ from 1 to $p$}{
    $q_{i,j} \gets \iota(C_j[X_i;\;1]) +\sigma(\epsilon$)\;}}
    $Y_1 \gets \underset{{X\in \RR^3}}{\mathrm{argmin}} \sum^{m}_{j=1} (q_{1,j}-\iota(C_j[X;\;1]))^2$\;
    $Y_2 \gets \underset{{X\in \RR^3}}{\mathrm{argmin}} \sum^{m}_{j=1} (q_{2,j}-\iota(C_j[X;\;1]))^2$\;
    $L_0\gets \mathrm{span}\{[Y_1;\;1], [Y_2;\;1]\}$\;
    \For{$i$ from 3 to $p$}{
    $Y_i \gets \underset{{X\in \RR^3: [X;\;1]\in L_0}}{\mathrm{argmin}} \sum^{m}_{j=1} (q_{i,j}-\iota(C_j[X;\;1]))^2$\; 
    }
    $e\gets \log_{10}\left(\frac{1}{p\epsilon}\sum_{i=1}^p \|Y_i-X_i\|\right) $\;
\Return{$e$}

 \end{algorithm}

 \begin{algorithm}
\SetKwInOut{Input}{Input}\SetKwInOut{Output}{Output}
\SetKwInOut{Return}{Return}
\caption{Method (L1).3 std.}
\label{alg:L1.3}
\Input{$\mathcal{C}=(C_1,\ldots,C_m)$, $L$,$X_1,\ldots,X_p$}
\Output{The log of the average relative error}
    \For{$j$ from 1 to $m$}{
    \For{$i$ from 1 to $p$}{
    $q_{i,j} \gets \iota(C_j[X_i;\;1]) +\sigma(\epsilon)$\;}
    $u_j \gets \iota(C_j \cdot L) +\sigma(\epsilon)$\;}
    $Y_1 \gets \underset{{X\in \RR^3}}{\mathrm{argmin}} \sum_{j=1}^m(q_{1,j}-\iota(C_j[X;\;1]))^2$\;
    $L_0 \gets \underset{{L': [L';\;1]\in \Lambda(Y_1)}}{\mathrm{argmin}} \sum^{m}_{j=1} (u_j-\iota(C_j \cdot [L';\;1]))^2$\;
    \For{$i$ from $2$ to $p$}{
    $Y_i \gets \underset{{X\in \RR^3:[X;\;1]\in L_0}}{\mathrm{argmin}} \sum^{m}_{j=1} (q_{i,j}-\iota(C_j[X;\;1]))^2$\;}
    $e\gets \log_{10}\left(\frac{1}{p\epsilon}\sum_{i=1}^p \|Y_i-X_i\|\right)  $\;

\Return{$e$}

 \end{algorithm}

 \begin{algorithm}
\SetKwInOut{Input}{Input}\SetKwInOut{Output}{Output}
\SetKwInOut{Return}{Return}
\caption{Method (L1).4.}
\label{alg:L1.4}
\Input{$\mathcal{C}=(C_1,\ldots,C_m)$, $L$, $X_1,\ldots,X_p$}
\Output{The log of the average relative error} 

    \For{$j$ from $1$ to $m$}{
    \For{$i$ from $1$ to $p$}{
        $q_{i,j} \gets \iota(C_j[X_i;\; 1]) +\sigma(\epsilon)$\;
        }
        $u_j \gets \iota(C_j \cdot L) +\sigma(\epsilon)$\;
        }
    $L_0 \gets \underset{{L': [L';\; 1]\in \mathrm{Gr}(1,\PP^3)}}{\mathrm{argmin}} \sum^{m}_{j=1} (u_j-\iota(C_j\cdot [L';\; 1]))^2$\;
    \For{$i$ from $1$ to $m$}{$Y_i \gets \underset{{X\in \RR^3: [X;\; 1]\in L_0}}{\mathrm{argmin}} \sum^{m}_{j=1} (q_{i,j}-\iota(C_j[X;\; 1]))^2$\;}
    $e\gets \log_{10}\left(\frac{1}{p\epsilon}\sum_{i=1}^p \|Y_i-X_i\|\right) $\;
\Return{$e$}

 \end{algorithm}

 \begin{algorithm}
\SetKwInOut{Input}{Input}\SetKwInOut{Output}{Output}
\SetKwInOut{Return}{Return}
\caption{Method (L1).1}
\label{alg:L1.1 nonstd}
\Input{$\mathcal{C}=(C_1,\ldots,C_m)$, $L$, $X_1,\ldots,X_p$}
\Output{The log of the average relative error} 

    \For{$j$ from $1$ to $m$}{
    \For{$i$ from $1$ to $p$}{
        $q_{i,j} \gets \iota(C_j[X_i;\;1]) +\sigma(\epsilon)$\;}
        $u_j \gets \iota(C_j \cdot L) +\sigma(\epsilon)$\;
    }
    $l_0,l_1\gets \textnormal{ON-basis of } \mathrm{nullspace} \begin{bmatrix}
    C^T_1[u_1;\; 1] & C^T_2[u_2;\; 1]
    \end{bmatrix}^T$\;
    \For{$j$ from $1$ to $m$}{
        $a_{j1},a_{j2}\gets \textnormal{ON-basis of } \begin{bmatrix}C_jl_0 & C_jl_1\end{bmatrix}^T$ \;
        $A_j\gets \begin{bmatrix}a_{j1} & a_{j2}\end{bmatrix}^T $\;
        $C_{\mathrm{aug},j}\gets  A_jC_j\begin{bmatrix}l_0 & l_1\end{bmatrix}$\;
        \For{$i$ from $1$ to $p$}{
        $q_{i,j}^A\gets \iota(A_jq_{i,j})$\;}
    }
    \For{$i$ from $1$ to $p$}{
        $Y_i' \gets \underset{{X\in  \RR^1}}{\mathrm{argmin}} \sum^{m}_{j=1} (q_{i,j}^A-\iota(C_{\mathrm{aug},j}[X;\;1]))^2$\;
        $Y_i\gets \iota(\begin{bmatrix}l_0 & l_1\end{bmatrix}[Y_i';\; 1])$\;
    }
    $e\gets \log_{10}\left(\frac{1}{p\epsilon}\sum_{i=1}^p \|Y_i-X_i\|\right) $\;
\Return{$e$}

 \end{algorithm}

\end{document}